\documentclass{birkjour}
\usepackage{csquotes}
\usepackage[many]{tcolorbox}
 \newtheorem{thm}{Theorem}[section]
 \newtheorem{cor}[thm]{Corollary}
 \newtheorem{lem}[thm]{Lemma}
 \newtheorem{prop}[thm]{Proposition}
 \newtheorem{problem}[thm]{Problem}
 \theoremstyle{definition}
 \newtheorem{defn}[thm]{Definition}
 \theoremstyle{remark}
 \newtheorem{rem}[thm]{\bf Remark}
 \newtheorem{ex}[thm]{\bf Example}
 \numberwithin{equation}{section}
 \newtheorem{quest}[thm]{\bf Question}

\begin{document}

\title[The Local Operator Moment Problem on  $\mathbb{R}$]
 {The Local Operator Moment Problem on  $\mathbb{R}$}

\author{R.E. Curto}
\address{R.E. Curto, Department of Mathematics, The University of Iowa, Iowa City, Iowa, U.S.A.}
\email{raul-curto@uiowa.edu}

\author[A. Ech-charyfy]{A. Ech-charyfy}

\address{%
Laboratory of Mathematical Analysis and Applications,\\
Faculty of Sciences, Mohammed V University in Rabat,\\
4 Av. Ibn Battouta, Rabat,
Morocco}

\email{abderrazzak\_echcharyfy@um5.ac.ma}

\author[H. El Azhar]{H. El Azhar}
\address{Laboratoire d’Informatique, Mathématique et leurs Applications (LIMA)\\ Faculty of Sciences, Chouaib Doukkali University,\\
Route Ben Maachou, 24000, El Jadida, Morocco}
\email{el-azhar.h@ucd.ac.ma}

\author[E.H. Zerouali]{E.H. Zerouali}
\address{Permanent address: Laboratory of Mathematical Analysis and Applications, Faculty of Sciences, Mohammed V University in Rabat, Rabat, Morocco.\& Department of Mathematics, The University of Iowa, Iowa City, Iowa, U.S.A.}
\email{elhassan.zerouali@fsr.um5.ac.ma \& ezerouali@uiowa.edu}


\begin{abstract}
We study the connections between operator moment sequences ${\mathcal T}=\displaystyle(T_n)_{n\in\mathbb{Z}_+}$ of self-adjoint operators on a complex Hilbert space $\mathcal{H}$ and the local moment sequences $\langle{\mathcal T}x,x\rangle = (\langle T_nx,x\rangle)_{n\in\mathbb{Z}_+}$  for arbitrary $x\in \mathcal{H}$. \ We provide necessary and sufficient conditions for solving the operator moment problem on $\mathbb{R}$, and we show that these criteria are automatically valid on compact subsets of $\mathbb{R}$. \ Applications of the compact case are used to study subnormal operator weighted shifts. \ A Stampfli-type propagation theorem for subnormal operator weighted shifts is also established. \ In addition, we discuss the validity of Tchakaloff's Theorem for operator moment sequences with compact support.\ In the case of a recursively generated sequence of self-adjoint operators, necessary and sufficient conditions for an affirmative answer to the operator recursive moment problem are provided, and the support of the associated representing operator-valued measure is described.

\end{abstract}

\maketitle
 
\section{Introduction}
In the sequel, we use the symbols $\mathbb{C}$, $\mathbb{R}$,   $\mathbb{N}=\{1,2,\cdots\}$ and $\mathbb{Z}_+=\{0,1,\cdots\}$ to denote the sets of complex numbers, real numbers,  positive integers, and non negative integers respectively. \ For $\mathbb{K}=\mathbb{C} \mbox{ or } \mathbb{R}$, $\mathbb{K}[X]$ represents the algebra of all polynomials with coefficients in $\mathbb{K}$. \ Denote ${\mathcal Z}(P)$ for the set of zeros of the polynomial $P\in \mathbb{R}[X]$. \ Given  $K\subseteq \mathbb{R}$, we denote by $\mathcal{P}(K)$ the algebra of all complex polynomials over $K$, and by $\mathcal{P}^+(K)$ the cone of positive polynomials of $\mathcal{P}(K)$. \ For $K$  compact,   $\mathcal{C}(K)$ is the $C^*$--algebra of all continuous complex functions on $K$ endowed with its supremum norm.

Let $\mathcal{H}$ and $\mathcal{K}$ be separable (complex) Hilbert spaces. \ Denote by $\mathbf{B}(\mathcal{H}, \mathcal{K})$ the Banach space of all bounded linear operators from $\mathcal{H}$ to $\mathcal{K}$. \ If $T \in \mathbf{B}(\mathcal{H}, \mathcal{K})$, then $\mathcal{N}(T)$ and $\mathcal{R}(T)$ stand for the kernel and the range of $T$, respectively. \ It is well known that $\mathbf{B}(\mathcal{H}):=\mathbf{B}(\mathcal{H}, \mathcal{H})$ is a $C^*$--algebra with   unit $I_\mathcal{H}$, the identity operator on $\mathcal{H}$, and $0_\mathcal{H}$ the null operator on $\mathcal{H}$. \  We will refer to the inner product and the corresponding norm of $\mathcal{H}$ by $\langle \cdot, \cdot \rangle_{\mathcal{H}}$ and $\| \cdot \|_{\mathcal{H}}$, respectively, or simply by $\langle \cdot, \cdot \rangle$ and $\| \cdot \|$, if there is no possible confusion. \ The unit sphere of $\mathcal{H}$ is given by $\mathcal{S}_{\mathcal{H}} = \{x \in \mathcal{H} \mid \|x\|_{\mathcal{H}} = 1\}$. \ As usual, for $T\in \mathbf{B}(\mathcal{H})$, the adjoint operator is denoted by $T^*$ and the spectrum of $T$  by $\sigma(T)$. \ We say that $T$ is {\it self-adjoint} if $T=T^*$, is {\it positive} if $\langle Tx,x \rangle_{\mathcal{H}}\ge 0$ for all $x\in\mathcal{H}$, and  is an {\it orthogonal projection} if it is self-adjoint and $T^2=T$. \ We shall often use the notation $T\ge 0$ for a positive operator $T$. \ The real vector space of all bounded self-adjoint operators is denoted by $\mathbf{B}_h(\mathcal{H})$, and the cone of all positive linear operators is denoted by $\mathbf{B}_+(\mathcal{H})$. \ The algebra of all sequences ${\mathcal T}=(T_n)_{n\in\mathbb{Z}_+}$ where $T_n\in\mathbf{B}_h(\mathcal{H}) $ for every $n\in\mathbb{Z}_+$ is denoted by  $\mathbf{B}_h(\mathcal{H}) ^{\mathbb{Z}_+}$. \ A sequence ${\mathcal T}=(T_n)_{n\in\mathbb{Z}_+}$ in  $\mathbf{B}_h(\mathcal{H}) ^{\mathbb{Z}_+}$ will be said to be {\it positive} (in symbols, ${\mathcal T}\ge 0$) if for every $x_0,\dots,x_n \in \mathcal{H}$ and $n\in\mathbb{Z}_+$, we have $\sum\limits_{i,j =0}^{n}\langle T_{i+j}x_i,x_j\rangle_{\mathcal{H}}\ge 0$. 

We use $\mathcal{B}(\mathbb{R})$ to denote the $\sigma$--algebra of all Borel subsets of $\mathbb{R}$. \ An {\it operator-valued charge} (abbreviated OVC) is a mapping $E:\mathcal{B}(\mathbb{R})\to \mathbf{B}(\mathcal{H})$, such that for every $x,y\in\mathcal{H}$, the function $E_{x,y}(.)=\langle E(.)x,y\rangle_{\mathcal{H}}$ is a complex Borel measure on $\mathbb{R}$ such that $E_{x,x}(.)$ is a real charge for every $x\in \mathcal{H}$. \ Here by a {\it real charge} on $\mathbb{R}$ we mean a countably additive map $\mu:\mathcal{B}(\mathbb{R}) \to \mathbb{R}$.\\
 
An {\it operator-valued measure} (abbreviated OVM) is an operator-valued charge $E$ such that $E_{x,x}(.)$ is a positive measure for every $x\in \mathcal{H}$. \
As usual, for every $\lambda \in \mathbb{R}$,  the scalar measure concentrated on the singleton set  $\{\lambda\}$ is denoted by $\delta_\lambda$.  

In the literature, for a sequence ${\mathcal T}=(T_n)_{n\in\mathbb{Z}_+}\in\mathbf{B}_h(\mathcal{H}) ^{\mathbb{Z}_+}$, the term ``operator moment sequence on $K$" typically refers to the existence of integral representations in the operator moment form \eqref{1} below, as in the following definition.
\begin{defn}
A sequence ${\mathcal T}=(T_n)_{n\in\mathbb{Z}_+}\in\mathbf{B}_h(\mathcal{H}) ^{\mathbb{Z}_+}$ is called an operator $K$--moment sequence, where  $K\subseteq \mathbb{R}$ is closed, if there exists an operator-valued measure $E$ supported on $K$, such that
\begin{equation}\label{1}
\langle T_nx,y\rangle_{\mathcal{H}}=\int_{K}t^n dE_{x,y}(t), \text{ for every } x,y\in\mathcal{H} \text{ and }n\in\mathbb{Z}_+.
\end{equation}
In this case, $E$ is a {\it representing operator-valued measure} for the operator moment sequence  ${\mathcal T}$.
\end{defn} We also define a weaker version of the operator moment problem as follows.
\begin{defn}
A sequence ${\mathcal T}=(T_n)_{n\in\mathbb{Z}_+}\in\mathbf{B}_h(\mathcal{H}) ^{\mathbb{Z}_+}$ is called a {\it local operator $K$--moment sequence}, where  $K\subseteq \mathbb{R}$ is closed, if for every $x\in {\mathcal H}$ we have $\langle{\mathcal T}x,x\rangle:=(\langle T_nx,x\rangle)_{n\in\mathbb{Z}_+}$ is a scalar moment sequence. \ That is if for every $x\in {\mathcal H}$ there exists a positive measure $\mu_x$ supported on $K$, such that
\begin{equation}\label{el1}
\langle T_nx,x\rangle_{\mathcal{H}}=\int_{K}t^n d\mu_{x}(t).
\end{equation}
An operator moment sequence (resp. a local operator moment sequence) is {\it determinate} if there is a unique operator-valued measure satisfying \eqref{1} (resp. a unique $\mu_x$ satisfying \eqref{el1} for each $x\in\mathcal{H}$) and is {\it indeterminate} otherwise. \ An OVM associated with an indeterminate (determinate) operator moment sequence is referred to as indeterminate (determinate) OVM.
\end{defn} 
Determining whether a given sequence ${\mathcal T}=(T_n)_{n\in\mathbb{Z}_+}$ is an operator $K$--moment sequence is a challenging task and requires sophisticated mathematical techniques. \ The operator $K$--moment problem is formally stated as follows.
\begin{problem}[Operator moment problem]\label{omp}  Let ${\mathcal T}=(T_n)_{n\in\mathbb{Z}_+}\in\mathbf{B}_h(\mathcal{H}) ^{\mathbb{Z}_+}$ and let $K$ be a closed subset of $\mathbb{R}$. \ Find necessary and sufficient conditions to ensure that ${\mathcal T}$ is an operator $K$--moment sequence.
\end{problem} 
Let ${\mathcal T}=(T_n)_{n\in\mathbb{Z}_+}\in\mathbf{B}_h(\mathcal{H}) ^{\mathbb{Z}_+}$ be an operator moment sequence and let $E$ be an associated representing measure. \ It is clear that  for every $x\in \mathcal{H}$, the sequence $\langle{\mathcal T}x,x\rangle:=(\langle T_nx,x\rangle_{\mathcal{H}})_{n\in\mathbb{Z}_+}$ is a scalar moment sequence associated with the scalar measure $E_x(.):=\langle E(.)x,x\rangle$. \ It is then natural to consider the next problem, which we will refer to as the {\it local operator $K$--moment problem}.
\begin{problem}[Local operator moment problem]\label{lmp}  Let ${\mathcal T}=(T_n)_{n\in\mathbb{Z}_+}\in\mathbf{B}_h(\mathcal{H}) ^{\mathbb{Z}_+}$ and let $K$ be a closed subset of $\mathbb{R}$. \ Find necessary and sufficient conditions to ensure that ${\mathcal T}$ is a local operator $K$--moment sequence.
\end{problem} 

An affirmative answer to Problem \ref{omp} implies an affirmative answer to Problem \ref{lmp}. One of the first results on the converse was provided by B. Sz-Nagy in \cite{Nagy1952}. More precisely, if $K$ is compact, then Problem $\ref{lmp}$ owns a solution $\Rightarrow $ Problem \ref{omp} owns a solution.\ However, the converse is not valid in general, as shown in Example \ref{ex1} below. \ ( See Section 2  for further discussion and additional results. ) \\

In the truncated operator moment problem, the initial data consists of a finite sequence $(T_0,\cdots, T_r)$. \ The related moment problems are stated as follows.
\begin{problem}[Truncated operator moment problem]\label{tomp}   Given $  T_0,\cdots, T_r $ in $\mathbf{B}_h(\mathcal{H}) $ and a closed set $K\subseteq\mathbb{R}$. \ Under which conditions  there exists an operator-valued measure $E$ supported on $K$, such that
\begin{equation*}
\langle T_nx,y\rangle_{\mathcal{H}}=\int_{K}t^n dE_{x,y}(t), \text{ for every } x,y\in\mathcal{H} \text{  and }n=0,\cdots,  r.
\end{equation*}
\end{problem} 

\begin{problem}[Truncated local operator moment problem]\label{tlmp}  Given $ T_0,\cdots, T_r$ in $\mathbf{B}_h(\mathcal{H})$ and a closed set $K\subseteq\mathbb{R}$. \ Determine the conditions under which $ (\langle T_0x,x\rangle,\cdots, \langle T_rx,x\rangle)$  is a truncated scalar $K$--moment sequence for every $x\in\mathcal{H}$.
\end{problem} 

This paper is organized as follows. \ In Section 2, we introduce the essential tools for the problem. \ Operator-valued measures are defined, and their main spectral properties are given. \ We also investigate the solubility and the determinacy of the operator moment problem and the local operator moment problem.
The case of operator-valued moment problems on compact sets is studied in Section 3. \ Applications are given to subnormal operator weighted shifts, and an analog of Stampfli's Theorem is provided. \ Section 4 is devoted to the recursive operator moment problem.
\section{Operator-valued Moment Problems}
\subsection{Operator-valued measures}
The notion of ``positive" operator-valued measure (abbreviated OVM), introduced in the 1940s by Naimark \cite{Neumark}, represents an extension of the concept of spectral measure. \ To better understand this notion, it is essential to note that spectral measures are in a one-to-one correspondence with self-adjoint operators. \ Spectral measures hold profound significance in quantum physics as they represent quantum observables, which describe measurable properties of quantum systems, such as the position or the spin of a particle. \ (For a thorough discussion on this topic, see, for instance,   \cite{brandt1999positive,moretti2017spectral}.) \ In other words, each spectral measure is associated with a unique self-adjoint operator. \ This correspondence is relatively straightforward. \ However, OVMs go even further since they generalize the concept of a quantum observable by allowing greater flexibility in describing measurements in quantum mechanics. \ Unlike spectral measures, OVMs are not strictly tied to self-adjoint operators but rather to symmetric operators. \ This generalization introduces additional complexity into the correspondence between OVMs and operators because symmetric operators are more varied and include self-adjoint operators as a particular case (see \cite{busch2016quantum,dubin2014operator}). \ In the sequel, we will define OVMs and describe their central role in the theory of operator moment problems.

An {\it operator-valued charge} (OVC for short) is a mapping $E:\mathcal{B}(\mathbb{R})\to \mathbf{B}(\mathcal{H})$, such that for every $x,y\in\mathcal{H}$, the function $E_{x,y}(.)=\langle E(.)x,y\rangle_{\mathcal{H}}$ is a complex Borel measure on $\mathbb{R}$. \ Moreover, the following statements hold.
\begin{enumerate}
\item Let $E$ be  an OVC. \ Then $E$ is an OVM  if and only if $E_x(.)=\langle E(.)x,x\rangle_{\mathcal{H}}$ is a positive measure for every $x\in\mathcal{H}$.
\item $E$ is a {\it semi-spectral measure} if it is an OVM satisfying $E(\mathbb{R})=I_\mathcal{H}$.
\item $E$ is a  {\it spectral  measure} if it is  {\it semi-spectral} and  $E(B)$ is an {\it orthogonal projection} for every $B\in \mathcal{B}(\mathbb{R})$.
\end{enumerate}

Let $E$ be an OVC. \ We define the support of  $E$ as the smallest closed subset $A$ of $\mathbb{R}$ such that $E(\mathbb{R}\setminus B)=0_\mathcal{H}$ for every Borel set $B\supseteq A$. \ (For a related definition, see the one presented at the end of \cite[Page 69]{schmudgen2012unbounded} or \cite[definition 16]{berberian1966notes}, in which the author designates the support of $E$ as the co-spectrum of $E$, i.e., $\bigvee (E)$.) 

The connection between the support of  $E$ and the supports of the corresponding charges $E_x \;\; (x\in\mathcal{H})$ is given by the following relation:

\begin{equation*}
\text{supp}(E)=\overline{\bigcup_{x, y \in\mathcal{H}} \text{supp}(E_{x,y})}.
\end{equation*}
Using the polarization identity, we obtain
\[E_{x,y}= \frac{1}{4}(E_{x+y}-E_{x-y} + iE_{x+iy}-iE_{x-iy}),\]
and then 
\begin{equation}\label{11}
\text{supp}(E)=\overline{\bigcup_{x\in\mathcal{H}} \text{supp}(E_x)}=\overline{\bigcup_{x\in\mathcal{S}_\mathcal{H}} \text{supp}(E_x)}.
\end{equation}
An OVC with finite support is said to be {\it finitely atomic}. \ It will be represented in the following form:
\begin{equation}\label{first}
E=\displaystyle\sum_{k=1}^{r}S_k\delta_{\lambda_k},
\end{equation}
where $\text{supp}(E)=\{\lambda_1,\lambda_2,\ldots,\lambda_r\}\subseteq \mathbb{R}$, and $S_1,S_2,\ldots,S_r \in \mathbf{B}(\mathcal{H})$. \\
It is clear that $E$ is an OVM if and only if $S_1, S_2, \ldots, S_r \in \mathbf{B}_+(\mathcal{H})$.

The relationship between an OVM, a {\it semi-spectral measure}, and a {\it spectral measure} is given by the following famous Naimark's Dilation Theorem.
\begin{thm}\label{pom}\cite[Theorem 4]{mlak1978dilations}
Let $E:\mathcal{B}(\mathbb{R})\to \mathbf{B}_+(\mathcal{H})$ be an {\rm OVM}. \ Then there exists a Hilbert space $\mathcal{K}$, a bounded linear operator $V:\mathcal{H}\to\mathcal{K}$, and a spectral measure $F:\mathcal{B}(\mathbb{R})\to \mathbf{B}_+(\mathcal{K})$ such that
\begin{equation*}
E(.)=V^*F(.)V.
\end{equation*}
Moreover, if $E$ is a semi-spectral measure, then $\mathcal{H}\subseteq\mathcal{K}$ and $E(.)=P_\mathcal{H}F(.)$ on $\mathcal{H}$, where $P_\mathcal{H}$ is the orthogonal projection of $\mathcal{K}$ onto $\mathcal{H}$.
\end{thm}
A recent criterion to determine when a {\it semi-spectral measure} is {\it spectral} has been obtained by P. Pietrzycki and J. Stochel in \cite[Theorem 4.2]{pietrzyckistochelJFA} and \cite{pietrzycki2022two}.

Let $E$ be an OVM supported on  $K\subseteq\mathbb{R}$  and $n\in\mathbb{Z}_+$. \
The integral of a measurable function \(f : \mathbb{C} \rightarrow \mathbb{C}\), denoted
$
\int_{K} f(\lambda)  dE(t) \in \mathbf{B}(\mathcal{H})
$
is defined by:
\[
\langle\int_{K} f(t)  dE(t)x,\ y\rangle_{\mathcal{H}} =
\int_{K} f(t) \langle dE(t)x, y \rangle_{\mathcal{H}}
\]
for arbitrary \(x, y \in \mathcal{H}\), provided all integrals on the right-hand side converge. \ When the integral converges, the operator
\begin{equation}\label{om}
T_n = \displaystyle\int_{K}t^ndE(t)\in \mathbf{B}_h(\mathcal{H})
\end{equation}
is called the $n^{th}$ {\it operator moment} of $E$. 

In the case where $E$ is {\it finitely atomic}, given by \eqref{first},  the $n^{th}$ operator moment is:
\begin{equation*}
T_n = \displaystyle\int_{K} t^n dE(t) = \sum_{k=1}^{r} \lambda_{k}^{n} S_k.
\end{equation*}
For $P(X)=\displaystyle\prod_{k=1}^{r}(X-\lambda_k)=X^r-a_{r-1}X^{r-1}-\dots-a_{0}$, the integral of $t^nP(t)$ with respect to $E(t)=\displaystyle\sum_{k=1}^{r}S_k\delta_{\lambda_k}(t)$ yields the null operator. \ Hence, the following recursive relation holds:
\begin{equation}\label{rr}
T_{n+r}=a_{r-1}T_{n+r-1}+\dots+a_{0}T_{n} \text{ for every } n\in\mathbb{Z}_+.
\end{equation}
A  sequence ${\mathcal T}$ satisfying \eqref{rr} is called a {\it recursive sequence}. \ The associated operator moment problem will be called a {\it recursive operator moment problem}. \ In this case, several classical questions can be addressed. \ Among these questions, we pose the following:
\begin{quest}\label{q3} Is there any simple  characterization for the existence and uniqueness of a representing measure for $(T_n)_{n\in\mathbb{Z}_+}\in\mathbf{B}_h(\mathcal{H}) ^{\mathbb{Z}_+}$ satisfying \eqref{rr}?\end{quest}
\begin{quest}\label{q4} Can we give explicitly the associated representing measure?\end{quest}
\begin{quest}\label{q5} Is every operator moment sequence a limit of recursive operator moment sequences (as in the scalar case)?\end{quest}
\subsection{Operator Moment Problems and Local Operator Moment Problems}
We start this section with some examples that illustrate the difference between the operator moment problem and the local operator moment problem. \ An example of a local operator moment sequence that is not an operator moment sequence is given.
\begin{ex}
Let $\mathcal{H} = L^2([0,1])$ be the complex Hilbert space of square-integrable functions defined on the interval $[0,1]$. 

Consider the sequence of self-adjoint operators  ${\mathcal T} = (T_n)_{n\in\mathbb{Z}_+}$ defined  for every $n\in\mathbb{Z}_+$ on $\mathcal{H}$ by $T_n(f)(t) = t^n f(t)$ with  $f \in L^2([0,1])$ and  $t \in [0,1]$. \ For $f, g\in \mathcal{H}$, we have $\langle T_nf, g \rangle_{\mathcal{H}} = \int_{0}^{1} t^n f(t) \overline{g(t)} dt = \int_{0}^{1} t^n dE_{f,g}(t)$, where $E_{f,g}(B) = \langle E(B)f, g \rangle_{\mathcal{H}} = \int_{B} f(t) \overline{g(t)} dt$ for any Borel set $B \subseteq [0,1]$. \ Thus, the sequence ${\mathcal T} = (T_n)_{n\in\mathbb{Z}_+}$ is an operator moment sequence on $[0,1]$ and the corresponding representing OVM is the {\it spectral measure} given by $E(B): f \to  \chi_B \cdot f$, where $\chi_B$ is the characteristic function of the Borel set $B \subseteq [0,1]$.\\ For $\phi \in L^2([0,1])$, the positive measure associated with the  derived local operator moment sequence $\langle{\mathcal T}\phi,\phi\rangle  $ is given by,
$dE_\phi(t) =|\phi|^2dt$.
\end{ex}
\begin{ex}\label{ex2.5}
Let  $T\in\mathbf{B}_h(\mathcal{H}) $ and let $E$ be its representing spectral measure. \ The sequence ${\mathcal T} = (T^n)_{n\in\mathbb{Z}_+}$ is  an operator moment sequence on $\sigma(T)$ and equivalently ${\mathcal T}$ is  a local operator moment sequence on $\sigma(T)$. This last fact will follow from a more general result on local operator moment sequences supported in compact sets; see Theorem \ref{main 1} below.
\end{ex}
In Example \ref{ex2.5}, the sequence ${\mathcal T} = (T^n)_{n\in\mathbb{Z}_+}$ is an operator moment sequence if and only if it is a local operator moment sequence. \ Thus, the following question in connection with Problems \ref{omp} and \ref{lmp} arises naturally. 
\begin{quest}
Does every local operator moment sequence derive from an operator moment sequence?
\end{quest}
The next example in $\mathbf{B}({\mathbb C}^2) $ given in \cite[Theorem 1]{bisgaard1994positive} provides a negative answer to the previous question and shows that additional assumptions are needed to get a positive answer.
\begin{ex}\label{ex1} Let $T_0=\begin{pmatrix}
    4&0\\0&1
\end{pmatrix}$, $T_1=\begin{pmatrix}
    0&2\\2&0
\end{pmatrix}$,  $T_2=\begin{pmatrix}
    1&0\\0&4
\end{pmatrix}$ and for $n\ge 2$,  we consider $T_{2n-1}=0_2$ and $T_{2n}=2^{(n+2)!}1_2$, where $0_2$ and $1_2$ are the null $2\times2$--matrix and the identity $2\times2$--matrix on  ${\mathbb C}^2.$ \ It is easily verifiable that $\begin{pmatrix}
    T_0&T_1\\T_1&T_2
\end{pmatrix}
$ is not positive  and then  ${\mathcal T}=(T_n)_{n\in\mathbb{Z}_+}$ is not a {\it matrix moment sequence}. \ On the other hand, a direct computation of the successive determinants of the Hankel matrix associated with  $\langle{\mathcal T}x,x\rangle$  shows that $\langle{\mathcal T}x,x\rangle$ is a {\it scalar moment sequence} for every $x\in {\mathbb C}^2$. \ Hence ${\mathcal T}$ is a {\it local operator moment sequence}.    
\end{ex}

To build an {\it operator moment sequence} from a {\it local operator moment sequence}, a necessary and sufficient condition is given in the following lemma, which is a particular case of \cite[Theorem 2]{berberian1966notes}. 
\begin{lem}\label{22}
Let ${\mathcal T}=(T_n)_{n\in\mathbb{Z}_+}$ be a local operator moment sequence  and  $(\mu_x)_{x\in\mathcal{H}}$ be the family of the associated positive measures given by \eqref{el1}. \ Then,  ${\mathcal T}$ is an operator moment sequence  associated with  an operator-valued measure $E$ such that
\begin{equation}\label{3}
\mu_x(.)=\langle E(.)x,x\rangle_{\mathcal{H}} \quad \text{for } x\in\mathcal{H},
\end{equation}
if and only if the following conditions hold:
\begin{equation}\label{6}
\mu_{x+y} + \mu_{x-y} = 2 \mu_x + 2 \mu_y
\end{equation}
\begin{equation}\label{4}
 [\mu_{x + y}]^{\frac{1}{2}} \le [\mu_{x}]^{\frac{1}{2}}+[\mu_{y}]^{\frac{1}{2}} 
\end{equation}
\begin{equation}\label{2.8}
\mu_{cx} = |c|^2 \mu_x
\end{equation}
(for all $x,y\in\mathcal{H}$ and all $c\in\mathbb{C}$), and, for each $B\in\mathcal{B}(\mathbb{R})$, there exists a positive constant $k_B$ such that
\begin{equation}\label{7}
\mu_x(B) \leq k_B ||x||_\mathcal{H}^{2} \quad \text{for all } x\in\mathcal{H}.
\end{equation}
In this case,   \eqref{3} determines $E$ uniquely.
\end{lem}
\begin{rem}
   The general families of $(\mu_x)_{x\in\mathcal{V}}$, where $\mathcal{V}$ is a linear space endowed with a seminorm satisfying parallelogram law has been considered in \cite{Szafran}. \ A general selection result as in Lemma \ref{22} is established. \ We refer to \cite{Szafran1, Szafran} for further results.
\end{rem}

\subsection{Determinate and indeterminate operator moment problems}
The {\it determinacy} or the {\it indeterminacy} of a {\it moment problem} has been a cornerstone question in the moment problem since its introduction. \ Various criteria to ensure determinacy can be found in the literature. \ See \cite{akhiezer2020classical} for example. \ A simple sufficient criterion was given by T. Carleman in the scalar case (cf. \cite{akhiezer2020classical}) and extended to operator sequences in \cite{bisgaard1994positive}, as described in the next proposition.
\begin{prop}\label{carle}
Let ${\mathcal T}=(T_n)_{n\in\mathbb{Z}_+}\in\mathbf{B}_h(\mathcal{H}) ^{\mathbb{Z}_+}$ be a local operator moment sequence on $\mathbb{R}$ such that
\begin{equation}\label{carl}
    \sum_{n=1}^{+\infty} \|T_{2n}\|^{-\frac{1}{2n}} = +\infty,
\end{equation}
then $\mathcal{T}$ is a determinate operator moment sequence.
\end{prop}

For every $x\in\mathcal{H}$ and $n\in\mathbb{Z}_+$, we observe that $|\langle T_nx,x\rangle|\le \|T_{n}\|\|x\|^2$, and obtain that if $\mathcal{T}$ is an operator moment sequence satisfying \eqref{carl}, then for every nonzero $x\in\mathcal{H}$ we have $\langle\mathcal{T}x,x\rangle$ satisfies the classical Carleman condition $$\sum_{n= 1}^{+\infty} \langle T_{2n}x,x\rangle^{-\frac{1}{2n}}=+\infty. $$ \ This allows us to derive the next result.
\begin{cor}
    Let $\mathcal{T}$ be an operator moment sequence satisfying \eqref{carl}. \ Then  $\mathcal{T}$ is a determinate local operator moment sequence.
\end{cor}
 \ The last result is not true in the general case if instead of condition \eqref{carl}, we assume that $\mathcal{T}$ is {\it determinate}. \ That is, there exists a {\it determinate} operator moment sequence $\mathcal{T }$, which is not a  {\it determinate local operator moment sequence}. \ This is the goal of the next proposition. 
 \begin{prop}\label{indeterm} There exists a determinate matrix moment sequence ${\mathcal T}=(T_n)_{n\in\mathbb{Z}_+}\in\mathbf{B}(\mathbb{C}^2)_{h}^{\mathbb{Z}_+}$ such that $\langle\mathcal{T }x,x\rangle$ is indeterminate for some $x\in\mathcal{S}_{\mathbb{C}^2}$.
\end{prop}
\begin{proof} In the context of \cite[Example (2.2)]{guardeno2001matrix}, let $\mu$  be an $N$--extremal measure of an indeterminate moment sequence. \ It is known that $\mu$ is discrete with mass in countably many points. \ It is also not difficult to see that, if $s, t \in supp(\mu)$ (each with positive mass), and if we let $\nu:=\mu(\{s\})\delta_s-\mu(\{t\})\delta_t$, then $ \mu \pm \nu$  are nonnegative determinate measures. \ From the identity
\[ \begin{pmatrix}
1 & -1 \\ 1 & 1
\end{pmatrix}\begin{pmatrix}
\mu & \nu \\ \nu & \mu
\end{pmatrix}\begin{pmatrix}
1 & 1 \\ -1 & 1
\end{pmatrix} = 2\begin{pmatrix}
\mu-\nu & 0 \\ 0 & \mu+\nu
\end{pmatrix}, \]
it follows that  
$E=\begin{pmatrix}
\mu & \nu \\ \nu & \mu
\end{pmatrix}$  is a determinate nonnegative matrix measure. \ 
Thus, if we denote $T_n=\displaystyle\int_\mathbb{R} t^ndE(t)$,  $ e_1 =(1,0)$, and $e_2=(0,1)$  we obtain
\begin{itemize}
    \item ${\mathcal T}=(T_n)_{n\in\mathbb{Z}_+}$ is a determinate operator moment sequence.
     \item $\langle\mathcal{T }e_1,e_1\rangle$ and $\langle\mathcal{T }e_2,e_2\rangle$  are indeterminate scalar moment sequences. 
     \item $\langle\mathcal{T }(e_1-e_2),e_1-e_2\rangle$ and $\langle\mathcal{T }(e_1+e_2),e_1+e_2\rangle$ are  determinate scalar moment sequences. 
\end{itemize}
\end{proof}
\begin{rem}
    Notice that there are bases that consist only of determinate scalar moment sequences and bases that consist only of indeterminate scalar moment sequences.\\
    Also, regarding Proposition \ref{carle}, using similarity, it is not difficult to check that Carleman's condition is not satisfied by the sequence ${\mathcal T}$.
\end{rem}
 
To derive the determinacy of an operator moment problem from the determinacy of the local one, a more precise condition is given in \cite{bisgaard1994positive}. \ We include this result for completeness.
\begin{prop}\cite[Theorem 2]{bisgaard1994positive}
    Let $\mathcal{T}$  be a local operator moment sequence. \ Suppose that 
    \begin{equation}\label{deterr}
    \langle\mathcal{T}x,x\rangle+\langle\mathcal{T}y,y\rangle 
    \mbox{ is determinate for every }    x, y\in \mathcal{H}, \end{equation} 
    then $\mathcal{T}$  is a determinate operator moment sequence.
\end{prop}
\begin{rem}
\begin{enumerate}
\item In Proposition \ref{indeterm}, we have $\langle\mathcal{T }(e_1-e_2),e_1-e_2\rangle$ and $\langle\mathcal{T }(e_1+e_2),e_1+e_2\rangle$ are  determinate while $\langle\mathcal{T }(e_1-e_2),e_1-e_2\rangle+\langle\mathcal{T }(e_1+e_2),e_1+e_2\rangle$ is   indeterminate. Thus, Condition \ref{deterr} is not true without additional assumptions.
    \item   Let $\{\mu_x, \: x\in \mathcal{H}\}$  be the family of scalar representing measures associated with $\mathcal{T}$. \ It is clear that  Condition \ref{deterr} is equivalent to the polar-type identity
    $$ \mu_{x+y} + \mu_{x-y}= 2(\mu_{x}+\mu_{y} ),$$ which allows us to construct a representing operator valued measure $E$. \ On the other hand, besides assuming the compactness of the support or a Carleman-type condition that ensures the validity of Condition \ref{deterr}, it is not easy to find other local moment sequences satisfying Condition \ref{deterr}.
\item The determinacy of a given operator moment sequence is inherited from the determinacy of the associated scalar moment sequences, as we show next. \ A local operator moment sequence  $\mathcal{T}$  is said to be {\it determinate} if for every $x\in \mathcal{H}$, the scalar moment sequence $\langle\mathcal{T}x,x\rangle$ is determinate.
\end{enumerate}

\end{rem}
 \begin{thm}\label{cdeter} Every operator moment sequence that is a determinate local operator moment sequence is a determinate operator moment sequence.
 
\end{thm}
\begin{proof}
 Let  ${\mathcal T}=(T_n)_{n\in\mathbb{Z}_+}\in\mathbf{B}_h(\mathcal{H}) ^{\mathbb{Z}_+}$ be an operator moment sequence that is a determinate local operator moment sequence and let $E, F$ be two associated representing OVMs. \ We have
  $$ \langle T_n x,y \rangle_{\mathcal{H}}=\displaystyle\int_\mathbb{R} t^n d \langle E x,y \rangle_{\mathcal{H}} (t)= \displaystyle\int_\mathbb{R} t^n d \langle F x,y \rangle_{\mathcal{H}} (t)$$ for every $x,y\in\mathcal{H}$.
  
  Now, since  $ \langle\mathcal{T }x,x\rangle:=\big(\big<T_nx,x\big>\big)_{n\in\mathbb{Z}_+}$ is determinate, we obtain $\langle E x,x \rangle_{\mathcal{H}}=\langle F x,x \rangle_{\mathcal{H}}$ for every $x\in {\mathcal{H}}$, and then  $E=F$.
  \end{proof}
\subsection{Solubility of the operator moment problem}

Solving the operator moment problem involves finding an appropriate representing OVM supported on $K$. \ Depending on the properties of the sequence ${\mathcal T}=(T_n)_{n\in\mathbb{Z}_+}$ and the set $K$, the problem can become quite sophisticated. \ It may require specialized mathematical tools and techniques to reach a conclusion. \ It is worth noting that operator moment sequences have significant applications in various fields; see the references cited in \cite[Introduction]{pietrzycki2022two}. \ Problem \ref{omp} can indeed be further formulated as the existence of integral representations for the following linear map on $\mathbb{C}[X]$: 
\begin{equation}\label{rf}
    L_{{\mathcal T}}: P = \displaystyle\sum_{k=0}^{n}a_{k}X^{k} \longmapsto  L_{{\mathcal T}}(P)=\displaystyle\sum_{k=0}^{n}a_{k}T_k  \in\mathbf{B}(\mathcal{H}).
\end{equation}
(In the scalar case, this map is known as the Riesz functional associated with the moment problem.) \  
The above formulation helps provide a more concrete representation of the problem and the connection between operator moment sequences and the integral representations via the mapping $L_{{\mathcal T}}$. \ (For more information on the integral representation of a given linear mapping, we refer to the work of J. Cimpri\v{c} and A. Zalar \cite{cimprivc2013moment}.)

The scalar case (when $\mathcal{H} = \mathbb{C}$) has been widely investigated. \ We refer to \cite{akhiezer2020classical} and \cite{schmudgen2017moment} where  a complete description of solutions can be found.

In the matrix case (when $\mathcal{H} =\mathbb{C}^n$, $n\ge 2$), which was introduced by Krein in \cite{krein1949fundamental}, the matrix $K$--moment problem becomes significantly more challenging and complex compared to the scalar case. \ This is due to the higher dimensionality and richer structure of the matrix space, which creates various difficulties in characterizing the solutions.
\ Numerous researchers have devoted their efforts to understanding the matrix case and addressing the associated challenges. \ Some notable works in this area includes \cite{adamyan2000solution},\cite{adamyan2006general},\cite{berg2008matrix},\cite{dyukarev2009distinguished},\cite{KimWoe}, \cite{kimsey2022solution}.

    In 1947,  M.G. Krein and M.A. Krasnosel'skii in \cite{krein1947fundamental} introduced the first operator $K$--moment problems in arbitrary Hilbert spaces, namely the Hausdorff ($K=\left[-1,1\right]$) and Stieltjes ($K=\left[0,+\infty\right[$)  operator moment problems. \ They also derived the solubility conditions for these two problems. \ For the Hausdorff operator moment problem, their results are formulated identically to the scalar case. \ In \cite[Theorem D]{krein1947fundamental}, it is  shown that a sequence ${\mathcal T}= (T_n)_{n\in\mathbb{Z}_+}$ of self-adjoint operators is an operator moment sequence if and only if
\begin{equation*}
    {\mathcal T}\ge 0 \ \ and \  \  \displaystyle\sum_{i,j=0}^{n}\langle (T_{i+j}-T_{i+j+2})x_i,x_j\rangle _{\mathcal{H}}\ge 0.
\end{equation*}
Likewise, for the Stieltjes operator moment problem, Krein and Krasnoselskii established  that a necessary and sufficient solubility condition is as follows:
\begin{equation*}
     {\mathcal T}\ge 0 \ \ and \  \  \displaystyle\sum_{i,j=0}^{n}\langle T_{i+j+1}x_i,x_j\rangle _{\mathcal{H}}\ge 0,
\end{equation*}
for every  $x_0,x_1,...,x_n\in \mathcal{H}$ and $n\in\mathbb{Z}_+$.

The condition ${\mathcal T}\ge 0$ is necessary and sufficient for the solubility of the Hamburger operator moment problem when the support of the representing measure is not subject to any conditions ($K=\mathbb{R})$; for instance, see \cite{lemnete2014hamburger}.

The case  of sequences of multi-operators  on semi-algebraic compact  sets  was set and solved by  F.-H. Vasilescu in \cite{vasilescu}.
\ As in the scalar multi-variable moment sequences, the solubility is given in term of positivity of the associated Reisz functional.

If $K$ is a compact subset of $\mathbb{R}$ and $T_0=I_\mathcal{H}$, the operator $K$--moment problem \ref{omp}  can be rephrased, equivalently, in terms of a self-adjoint dilation problem, by using Naimark's Dilation Theorem \ref{pom}, as follows.
\begin{problem}[Self-adjoint dilation problem]\label{dilation} 
Let ${\mathcal T}=(T_n)_{n\in\mathbb{Z}_+}\in\mathbf{B}_h(\mathcal{H}) ^{\mathbb{Z}_+}$ with $T_0=I_\mathcal{H}$. \
Find a bounded self-adjoint operator $B$ on a Hilbert space $\mathcal{K}\supseteq \mathcal{H}$ such that: 
    \(
        \sigma(B)\subseteq K
   \)
    and
    \(
        T_n=P_\mathcal{H}B^n \mbox{ for all } n\in \mathbb{Z}_+,
   \)
   where $P_\mathcal{H}$ represents the orthogonal projection from $\mathcal{K}$ onto $\mathcal{H}$.
\end{problem}

In 1953, Sz.-Nagy solved this problem by providing a necessary and sufficient condition, which can be expressed in the form of the positivity of the linear map $L_{{\mathcal T}}$ given by \eqref{rf}.
\begin{thm}\label{Nagy}\cite[Theorem 2]{Nagy1952}
Let $K\subseteq \mathbb{R}$ be a compact set. \ A sequence  ${\mathcal T}=(T_n)_{n\in\mathbb{Z}_+}\in\mathbf{B}_h(\mathcal{H}) ^{\mathbb{Z}_+}$ with $T_0=I_\mathcal{H}$  admits a self-adjoint operator dilation
$B$ if and only if  $L_{{\mathcal T}}(P)\ge 0$ for every $P\in \mathcal{P}^+(K)$.
\end{thm}
Consider now the case of ${\mathcal T}=(T_n)_{n\in\mathbb{Z}_+}\in\mathbf{B}_h(\mathcal{H}) ^{\mathbb{Z}_+}$, with invertible $T_0\ne I_\mathcal{H}$. \ Letting $\tilde{{\mathcal T}}:=(T_0^{-\frac{1}{2}}T_nT_0^{-\frac{1}{2}})_{n\in\mathbb{Z}_+}$ and observing that  $L_{{\mathcal T}}(P)\ge 0$  if and only if $L_{\tilde{{\mathcal T}}}(P)\ge 0$, we obtain, using the previous theorem, the following representation of ${\mathcal T}=(T_n)_{n\in\mathbb{Z}_+}\ge 0$ in the form:
\begin{equation}\label{111}
T_n=T_{0}^{\frac{1}{2}}P_\mathcal{H}B^nT_{0}^{\frac{1}{2}} \text{ for all } n\in\mathbb{Z}_+.
\end{equation}
Notice, however, that equation \eqref{111} does not give an integral representation of $T_0$.

\section{Local Operator Moment Sequences on Compact Sets}\label{section a}

As in the scalar case \cite[Proposition 4.1]{schmudgen2017moment}, the following proposition presents a characterization of OVMs supported on a compact interval $\left[-M, M\right]$ in relation to their operator moment sequences.
\begin{prop}\label{comp}
Let ${\mathcal T}=(T_n)_{n\in\mathbb{Z}_+}\in\mathbf{B}_h(\mathcal{H}) ^{\mathbb{Z}_+}$ be an operator moment sequence on $\mathbb{R}$ and $E$  a representing {\rm OVM} of ${\mathcal T}$. \ For  $M>0$, the following statements are equivalent:
\begin{enumerate}
    \item $\text{supp}(E)\subseteq [-M,M]$;
    \item There exists a constant $C>0$ such that $\|T_n\|\le CM^n$ for every $n\in\mathbb{Z}_+$;
    \item There exists a constant $C>0$ such that $\|T_{2n}\|\le CM^{2n}$ for every $n\in\mathbb{Z}_+$;
    \item L:=$\displaystyle\lim\inf_{n\to +\infty}\|T_{2n}\|^{\frac{1}{2n}}\le M.$
\end{enumerate}
\end{prop}
\begin{proof}
 $1 \Rightarrow 2$. \ For $x\in\mathcal{S}_{\mathcal{H}}$, we have $\langle T_nx,x\rangle_{\mathcal{H}} =\displaystyle \int_{[-M,M]} t^nd\langle Ex,x\rangle_{\mathcal{H}}(t)$, and therefore, $\|T_n\|= \sup\limits_{x\in \mathcal{S}_{\mathcal{H}} }|\langle T_nx,x\rangle_{\mathcal{H}}|  \leq C M^n$, where $C=\|T_0\|$.\\
  $2 \Rightarrow 3 \Rightarrow 4$. \ Clear.\\
 $4 \Rightarrow 1$. \ Let $x\in\mathcal{S}_\mathcal{H}$ and $n\in\mathbb{N}$. \ For every $\alpha >0$, we have
\[
\alpha^{2n}E_x([-\alpha, \alpha]^c) \leq \int_{[-\alpha, \alpha]^c}t^{2n}dE_x(t)\le \int_{\mathbb{R}}t^{2n}dE_x(t)= \langle T_{2n}x, x \rangle_{\mathcal{H}}.
\]
Which implies that
\[
\alpha^{2n}\sup_{x\in\mathcal{S}_\mathcal{H}}E_x([-\alpha, \alpha]^c) = \alpha^{2n}\|E([-\alpha, \alpha]^c)\| \leq \sup_{x\in \mathcal{S}_\mathcal{H}} \langle T_{2n}x, x \rangle_{\mathcal{H}} = \|T_{2n}\|.
\] Then,  $\alpha\|E([-\alpha, \alpha]^c)\|^\frac{1}{2n} \le\|T_{2n}\|^\frac{1}{2n}$. \
In particular, for every $\alpha$ such that  $\|E([-\alpha, \alpha]^c)\|>0$,  we obtain $\alpha \le L$. \ Therefore, $E([-\alpha, \alpha]^c)=0_\mathcal{H}$ when $\alpha > L$, which implies that $\text{supp}(E)\subseteq [-L,L]\subseteq [-M,M]$ and proves 1.
\end{proof}

With regard to Theorem \ref{cdeter}, the following   question   arises naturally:
\begin{quest}\label{q1u} Let ${\mathcal T} = (T_n)_{n\in\mathbb{Z}_+}$ be a determinate local operator moment sequence. \ Is it true that   the sequence $(T_n)_{n\in\mathbb{Z}_+}$ is an operator moment sequence?
 \end{quest}
In its full generality, we are not able to find an answer. \ In the case of compactly supported measures in which determinacy is automatic, we have the following theorem.
\begin{thm}\label{main 1}Let ${\mathcal T}=(T_n)_{n\in\mathbb{Z}_+}\in\mathbf{B}_h(\mathcal{H}) ^{\mathbb{Z}_+}$ and let $K$ be a compact set of $\mathbb{R}$. \ The following statements are equivalent:
\begin{enumerate}
    \item ${\mathcal T}$ is an operator $K$--moment sequence;
    \item ${\mathcal T}$ is a local operator $K$--moment sequence.
\end{enumerate}
It follows in particular that for every local operator $K$--moment sequence ${\mathcal T}$,  the scalar representing measures $(\mu_x)_{x\in\mathcal{H}}$  satisfy   \eqref{6}, \eqref{4} and  \eqref{2.8}.
\end{thm}
The implication $1 \Rightarrow 2$ is obvious. \ There are several ways to prove $2 \Rightarrow 1$, with the simplest one being to use the argument of Bisgaard \cite[Theorem 2]{bisgaard1994positive} or the results in the recent paper by Cimpri\v{c} and Zalar \cite[Theorem 4 and Remark 6]{cimprivc2013moment}.

Alternatively, since for a compactly supported operator moment sequence, we have $\|T_n\|\le \displaystyle\sup(K)^n\|T_0\|$, it follows that Carleman's Condition \eqref{carl} holds, and then $2 \Rightarrow 1$ is recovered by using Proposition \ref{carle}.

In addition, we will provide two other proofs of independent interest by using different approaches. \ The first one is based on the operator-valued version of F. Riesz's representation \cite[Theorem 19]{berberian1966notes}, while the second relies on the theory of completely positive maps described in \cite{bhat2023operator}.
\begin{proof}[Method 1]
Assume that condition 2 in Theorem \ref{main 1} is true. \ For  $x\in\mathcal{H}$, there is a positive finite measure $E_x$ supported on $K$ such that $\langle T_nx,x\rangle=\displaystyle\int_{K}t^ndE_x(t)$ \; for every $n\in\mathbb{Z}_+$. \
The linear map $L_x$ defined on $ \mathcal{P}(K)$ by  $L_x(X^n):=\langle T_nx,x\rangle_\mathcal{H}$ for all $n \in\mathbb{Z}_+$, satisfies
\begin{equation}\label{aa}
L_x(P)=\displaystyle\int_{K}P(t)dE_x(t)\ge 0 \quad \text{whenever } P\in \mathcal{P}^+(K).
\end{equation}
Now, we consider a linear map $L_{{\mathcal T}}:\mathcal{P}(K)\to\mathbf{B}_h(\mathcal{H})$ defined as $L_{{\mathcal T}}(X^n)=T_n$ for all $n \in\mathbb{Z}_+$. \ Then, \eqref{aa} is equivalent to
\begin{equation*}
L_{{\mathcal T}}(P)\ge 0  \text{ whenever } P\in \mathcal{P}^+(K).
\end{equation*}
 By the Stone-Weierstrass Theorem, $\mathcal{P}(K)$ is dense in $\mathcal{C}(K)$,  since $\mathcal{P}(K)$ is a $*$--subalgebra of $\mathcal{C}(K)$ and it separates points of $K$. \ As a result\footnote{From \cite[Exercise 2.2 page 21]{paulsen2002completely}.}, $L_{{\mathcal T}}$ can be extended to a positive linear map $\tilde{L}_{{\mathcal T}}$ on the $C^*$--algebra $\mathcal{C}(K)$, such that\footnote{From \cite[Theorem 2.4]{berberian1966notes}.} $\|\tilde{L}_{{\mathcal T}}(f)\|\le \|T_0\|\|f\|_{\infty}$ for every $f\in\mathcal{C}(K)$. \ Therefore, by \cite[Theorem 19]{berberian1966notes}, there exists an operator-valued measure $E$ supported on $K$ such that $\langle T_nx,x\rangle_{\mathcal{H}}=\int_{K}t^nd\langle Ex,x\rangle_{\mathcal{H}}(t)$ for every $x\in\mathcal{H}$. \
Moreover, by uniqueness, $E_x=\langle Ex,x\rangle_{\mathcal{H}}$ for every $x\in\mathcal{H}$.
\end{proof}
\begin{proof}[Method 2] By employing the first part of Method 1 and Stinespring's Theorem \cite[Theorem 3.11]{paulsen2002completely}, we can establish that the extension $\tilde{L}$ of $L$ is completely positive. \ Additionally, applying Stinespring's Dilation Theorem \cite[Theorem 4.1]{paulsen2002completely}, we can conclude the existence of a Hilbert space $\mathcal{K}$, a bounded operator $V:\mathcal{H}\to\mathcal{K}$, and a unital $*$--homomorphism $\pi:\mathcal{C}(K)\to\mathbf{B}(\mathcal{K})$ such that:
$L(f)=V^*\pi(f)V$, for all $f\in\mathcal{C}(K)$. \
Let $B=\pi(\chi)$, where $\chi(t):=t$ for every $t \in K$. \ Since $K\subseteq \mathbb{R}$, we have $B^*=\pi(\chi)^*=\pi(\chi^*)=\pi(\chi)=B$. \ This implies that $T_n=L(\chi^n)=V^*\pi(\chi^n)V=V^*\pi(\chi)^nV=V^*B^nV \mbox{ for all } n\in\mathbb{Z}_+.$ \ Now, let $F$ be a spectral measure of the adjoint operator $B$. \ Then, we have $\big<T_nx,y\big>_{\mathcal{H}}=\int_{K}t^nd\big<E x,y\big>_{\mathcal{H}}(t)
\mbox{ for all } x,y\in\mathcal{H} \mbox{ and } n\in\mathbb{Z}_+,$ where $E(\cdot)=V^*F(\cdot)V$ is a OVM  from Theorem \ref{pom}. \
It is clear that $supp(E)\subseteq supp(F)=\sigma(B)=K$. \ Thus, the proof is complete.
\end{proof}
\begin{rem}\label{rs}
\begin{enumerate}
\item  From the previous theorem, we observe that in Theorem \ref{main 1}, if $K$ is a compact set of $\mathbb{R}$, then the conditions \eqref{6}, \eqref{4}, and \eqref{2.8} are automatically satisfied. 
\item Using the notations of Method 2, when $T_0=I_\mathcal{H}$, the operator $V$ becomes an isometry. \ In this particular case, we can identify the Hilbert space $\mathcal{H}$ with the subspace $V(\mathcal{H})$ of $\mathcal{K}$. \ Consequently, under this identification, the adjoint operator $V^*$ corresponds to the orthogonal projection of $\mathcal{K}$ onto $\mathcal{H}$, represented as $P_\mathcal{H}$. \ As a result, we have the relationship $T_n=P_\mathcal{H}B^n$. \ Thus, we can conclude that the operator moment problem \ref{omp} on a compact set $K\subseteq \mathbb{R}$ is equivalent to the self-adjoint dilation problem (Problem \ref{dilation}).
\item  Now, we relax the condition that  ${\mathcal T}$ is a local operator moment sequence on prescribed compact set $K$ by replacing with the condition that $\langle{\mathcal T}x,x\rangle$ is an operator moment sequence on some compact $K_x$,  for every $x \in {\mathcal H}$. \ It will follow that $\langle{\mathcal T}x,x\rangle+\langle{\mathcal T}y,y\rangle$ is a determinate moment sequence for every $x,y \in {\mathcal H}$ and then ${\mathcal T}$ is a determinate  operator moment sequence. \ In this case, it would be interesting to know whether ${\mathcal T}$ necessarily admits a compactly supported representing measure. 
\end{enumerate}
\end{rem}
\section{Subnormal Operator Weighted Shifts and the Operator Moment Problem}
Let  $\mathcal{A} = \{A_n\}_{n=0}^{+\infty}$ be a sequence of non-negative invertible operators on $\mathcal{H}$ satisfying 
$\displaystyle\sup_{n\in\mathbb{Z}_+}\|A_n\|<+\infty$. 
\\The  operator  weighted shift  $W_\mathcal{A}$ associated with  $\mathcal{A}$ is defined  on the Hilbert space \( 
\ell^2(\mathcal{H}): = \left\{ (x_n)_{n\geq 0} \in \mathcal{H}^{{\mathbb Z}_+}  \mid \sum\limits_{n=0}^{+\infty} \|x_n\|_{\mathcal{H}}^2 < \infty \right\}\)   by:
\[
W_\mathcal{A}(x_0, x_1, \ldots) := (0, A_0x_0, A_1x_1, \ldots).
\]
The moment sequence of $W_\mathcal{A}$ is given by $\mathcal{B}_{\mathcal{A}}: = \{B_{n}^*B_n\}_{n=0}^{+\infty}$, where $B_0 = I_{\mathcal{H}}$ and $B_n = A_{n-1}B_{n-1}$ for $n\in 
 \mathbb{N}$. \  Conversely, given any sequence of invertible operators $(B_n)_{n\in\mathbb{Z}_+}$, there exists an operator weighted shift $W_\mathcal{A}$ such that $\mathcal{B} = \{B_{n}^*B_n\}_{n=0}^{\infty} = \mathcal{B}_{\mathcal{A}}$. \ It suffices to define $\mathcal{A}$ by the expression $A_n: = B_{n+1}B_{n}^{-1}$.\\

The moment sequence of an operator-weighted shift encodes several of its spectral properties as in the scalar case. \ In particular, we have the next useful characterization of subnormal operator weighted shifts provided in \cite[Theorem 1]{ghatage1976subnormal}. \ See also \cite{lambert1976subnormality} for further results.
\begin{thm}\label{sub} Let $W_\mathcal{A}$ be an operator weighted shift. \ Then, $W_\mathcal{A}$ is subnormal if and only if $\mathcal{B}_{\mathcal{A}}$ is an operator moment sequence on  $[0,\|W_\mathcal{A}\|^2].$
\end{thm}
Now, because $[0,\|W_\mathcal{A}\|^2]$ is compact, we apply  Theorem \ref{main 1} to deduce the following result.
\begin{prop}\label{sub1} Under the previous notations, the following statements are equivalent:
\begin{enumerate}
    \item $W_\mathcal{A}$ is subnormal;
    \item $\mathcal{B}_{\mathcal{A}}$ is a local operator moment sequence on  $[0,\|W_\mathcal{A}\|^2].$
\end{enumerate} 
\end{prop}
For every nonzero $x\in \mathcal{H}$, the scalar weighted shift $W_{\alpha(x)}$   on $\ell^2(\mathbb{C})$  associated with the  weight sequence $\alpha(x) = \left\{\frac{\|B_{n}x\|}{\|B_{n-1}x\|}\right\}_{n=1}^{+\infty}$ is bounded. \ Indeed, for   $x\in \mathcal{S}_\mathcal{H}$, we have $$ \frac{\|B_{n}x\|}{\|B_{n-1}x\|} = \frac{\|A_{n-1}B_{n-1}x\|}{\|B_{n-1}x\|} \le  \|A_{n-1}\|\le \displaystyle\sup_{m\in\mathbb{Z}_+}\|A_m\|<+\infty.$$    \ We now deduce the following theorem, which includes a new characterization of the subnormality of $W_\mathcal{A}$, based on Proposition \ref{sub1} and relying on the subnormality of $W_{\alpha(x)}$ for every nonzero $x\in \mathcal{H}$.
\begin{thm}\label{main 2}
Using the previous notations, the following statements are equivalent:
\begin{enumerate}
    \item $W_\mathcal{A}$ is subnormal;
    \item $W_{\alpha(x)}$ is subnormal for every nonzero $x\in \mathcal{H}$.
\end{enumerate}
\end{thm}
\begin{proof}
From Proposition \ref{sub1}, $W_\mathcal{A}$ is subnormal if and only if for every every nonzero $x\in \mathcal{H}$, the scalar sequence $(\langle B_n^*B_nx,x\rangle)_{n\in\mathbb{Z}_+}$ is a  moment sequence on $[0,\|W_\mathcal{A}\|^2]$. \ It follows that for $\alpha(x) = \{\alpha_n(x)\}_{n=1}^{+\infty}$ with $\alpha_n(x) =\displaystyle \sqrt{\frac{\langle B_n^*B_nx,x\rangle}{\langle B_{n-1}^*B_{n-1}x,x\rangle}} = \frac{\|B_{n}x\|}{\|B_{n-1}x\|}$, the shift operator $W_{\alpha(x)}$ is subnormal for every nonzero $x \in \mathcal{H}$.

\end{proof}
 Stampfli's Theorem established a propagation phenomenon for subnormal scalar weighted shifts, as follows.
\begin{thm}\cite[Theorem 6]{stampfli1966weighted} \label{stamp}
Let $\alpha = (\alpha_n)_{n\in\mathbb{Z}_+}$ be a sequence of positive real numbers. \ If $W_\alpha$ is a subnormal weighted shift, and $\alpha_k = \alpha_{k+1}$ for some $k \in \mathbb{Z}_+$, then $\alpha_n = \alpha_1$ for every $n\ge 1$.   
\end{thm}
We obtain the following local propagation results for operator weighted shifts as corollaries of Theorem \ref{main 2} and Theorem \ref{stamp}.
\begin{cor}
 Let $W_{\mathcal{A}}$ be a subnormal operator weighted shift and $x \in  {\mathcal{H}}$. \ If $\|B_kx\|^2 = \|B_{k-1}x\|\|B_{k+1}x\|$ for some $k\in \mathbb{N}$, then $\|B_nx\|^2 = \|B_{n-1}x\|\|B_{n+1}x\|$ for every $n \geq 1$.
 \end{cor}
\begin{cor} Let $W_\mathcal{A}$ be a subnormal operator weighted shift. \ Suppose that for every $x\in  {\mathcal H}$, there exists $k_x\ge 1$ such that  $\|B_{k_x}x\|^2=\|B_{k_x+1}x\|\|B_{k_x-1}x\|$, then $\|B_nx\|^2 = \|B_{n-1}x\|\|B_{n+1}x\|$ for every $x\in {\mathcal H}$ and $n\ge 1$.  
 \end{cor}
 The first-named author has extended Stampfli's Theorem to the more general class of $2$--hyponormal weighted shifts; see \cite{curt}, for more information. \ For matricial weighted shift  $W_\mathcal{A}$, when $\mathcal{A} = \{A_n\}_{n=0}^{+\infty}$  is  a sequence of non-negative invertible matrices on $\mathbb{C}^p$ and $p\in \mathbb{N} $, studied by N. Ivanovski in \cite{ivan}, a partial version of  Stampfli's Theorem  is obtained. More precisely, it is  shown that if $A_k = A_{k+1}$ for some $k \ge 1$, then $A_n =A_k$ for every $n\ge k$. \ In recent work \cite[Theorem 5.7]{cehz}, we have completed Ivanovski's result for matricial weighted shifts by showing that $A_n =A_k$ for every $n\ge 0$. \ Next, we give the extension of Stampfli's Theorem to the case of subnormal operator weighted shifts on $\mathcal{H}$. 
\begin{thm}\label{001}
Let  $\mathcal{A}=(A_n)_{n\in\mathbb{Z}_+}$ be a sequence of positive invertible
operators on $\mathcal{H}$ satisfying $\displaystyle\sup_{n\in\mathbb{Z}_+} \|A_n\|<+\infty$. \ If the operator weighted shift $ W_\mathcal{A}$ is  subnormal  such that $A_k = A_{k+1}$ for some $k\in\mathbb{Z}_+$, then $A_{n}= A_{1}$ for all $n\in\mathbb{Z}_+$. 
\end{thm}

 We will use the following version of Smul'jan's extension theorem for positive operators from \cite[Proposition 2.2]{curto}.
 \
\begin{lem}\label{smuljano} Let $X \in \mathbf{B}(\mathcal{H})$, $Z \in  \mathbf{B}(\mathcal{K})$, $X, Z \ge 0$, and $Y \in  \mathbf{B}(\mathcal{K}, \mathcal{H})$. \ The following statements are equivalent:
\begin{enumerate}
    \item $\begin{pmatrix}
        X & Y \\
        Y^* & Z
    \end{pmatrix} \ge 0$;
    \item There exists $W \in \mathcal{B}(\mathcal{K}, \mathcal{H})$ such that $X^{\frac{1}{2}}W = Y$ and $Z \ge W^*W$;
    \item There exists $U \in \mathcal{B}(\mathcal{H}, \mathcal{K})$ such that $Z^{\frac{1}{2}}U = Y^*$ and $X \ge U^*U$.
\end{enumerate}
\end{lem}
  We will use statement $2$ in Lemma \ref{smuljano} to prove that $A_{n} = A_{k}$ for all $n \geq k$ and $3$ to prove that $A_{n} = A_{k}$ for all $n \leq k$. 
  
  For $n \geq k$, without any loss of generality, we can assume that $k=0$ and hence that $A_0=A_1=:A$. \ We have
$$\begin{pmatrix}
    I & B_1^*B_1 & B_2^*B_2\\
    B_1^*B_1& B_2^*B_2 & B_3^*B_3\\
    B_2^*B_2& B_3^*B_3 & B_4^*B_4
\end{pmatrix}\ge 0 \Rightarrow \begin{pmatrix}
    I & A^2 &A^4\\
   A^2& A^4 & A^2A_2^2A^2\\
   A^4& A^2A_2^2A^2 & A^2A_2A_3^2A_2A^2\end{pmatrix}\ge 0.$$ 
\ Using statement 2 in Lemma \ref{smuljano}, with  $X=\begin{pmatrix}
       I & A^2 \\
   A^2& A^4 
   \end{pmatrix}$ and $Y=\begin{pmatrix}
       A^4  \\
  A^2A_2^2A^2 
   \end{pmatrix}$, there exists $W=\begin{pmatrix}
   W_1 \\
   W_2\\
  \end{pmatrix}$, such that \[\begin{pmatrix}
    I & A^2 \\
   A^2& A^4 \end{pmatrix}^{\frac{1}{2}}\begin{pmatrix}
   W_1 \\
   W_2\\
  \end{pmatrix}=\begin{pmatrix}
       A^4  \\
  A^2A_2^2A^2 
   \end{pmatrix}.\]  It is easy to check that $$\begin{pmatrix}
    I & A^2 \\
   A^2& A^4 \end{pmatrix}^{\frac{1}{2}}= \begin{pmatrix}
     (I+A^4)^{\frac{-1}{2}} & 0 \\
  0& (I+A^4)^{\frac{-1}{2}} \end{pmatrix}\begin{pmatrix}
    I & A^2 \\
   A^2& A^4 \end{pmatrix},$$ and then that
   $$ \left\{\begin{array}{lll}
      (I+A^4)^{\frac{-1}{2}}(W_1+A^2W_2) & =  & A^4  \\
         (I+A^4)^{\frac{-1}{2}}(A^2W_1+A^4W_2) & = & A^2A_2^2A^2. 
   \end{array}\right.
   $$ It follows that $A^6= A^2A_2^2A^2$  and hence $A_2=A$. \ We end the proof by induction.

To prove that $A_{n} = A_{k}$ for all $n \le k$, we can assume, without any loss of generality, that $k=2$ and we write $A_2=A_3=:A$. \ We need to show that $A_1=A$. \ To this aim,  we use the identity on $(X,X,X)$ to derive that
$$\begin{pmatrix}
    I & B_1^*B_1 & B_2^*B_2\\
    B_1^*B_1& B_2^*B_2 & B_3^*B_3\\
    B_2^*B_2& B_3^*B_3 & B_4^*B_4
\end{pmatrix}=\begin{pmatrix}
    I & A_0^2 &A_0A_1^2A_0\\
  A_0^2& A_0A_1^2A_0 & A_0A_1A^2A_1A_0\\
  A_0A_1^2A_0&  A_0A_1A^2A_1A_0&  A_0A_1A^4A_1A_0\end{pmatrix}\ge 0.$$ 
\ Using statement $3$ in Lemma \ref{smuljano}, with  $Z=\begin{pmatrix}
   A_0A_1^2A_0 & A_0A_1A^2A_1A_0\\
    A_0A_1A^2A_1A_0&  A_0A_1A^4A_1A_0\end{pmatrix}$ \; and \;  $Y=\begin{pmatrix}
      A_0^2  \\
 A_0A_1^2A_0
   \end{pmatrix}$, there exists $W=\begin{pmatrix}
   W_1 \\
   W_2
  \end{pmatrix}$ such that $$
  \begin{array}{lll}\begin{pmatrix}
   A_0A_1^2A_0 & A_0A_1A^2A_1A_0\\
    A_0A_1A^2A_1A_0&  A_0A_1A^4A_1A_0\end{pmatrix}^{\frac{1}{2}}\begin{pmatrix}
   W_1 \\
   W_2\\
  \end{pmatrix}&=&\begin{pmatrix}
      A_0^2  \\
 A_0A_1^2A_0
   \end{pmatrix}\end{array}$$
    Now, from the   identity $Z = \begin{pmatrix}
    A_0A_1  & 0\\
  0& A_0A_1 \\
 \end{pmatrix}\begin{pmatrix}
    I  & A^2\\
  A^2& A^4 \\
 \end{pmatrix}\begin{pmatrix}
    A_1A_0  & 0\\
  0& A_1A_0 \\
 \end{pmatrix}$, we derive  that $$\mathcal{R}(Z)=\{(A_0A_1V,A_0A_1A^2V) \mid V\in  \mathbf{B}(\mathcal{H})\}.$$ 
We also see that $Z\ge  0$  and that $Z$ has closed range. \ It follows that  $\mathcal{R}({Z}^\frac{1}{2})= \mathcal{R}(Z)=\{(A_0A_1V,A_0A_1A^2V) \mid V\in  \mathbf{B}(\mathcal{H})\}$ and hence that  there exists $V\in \mathbf{B}(\mathcal{H})$ such that 
 $$
 \left\{\begin{array}{ll}
       A_0^2 &= A_0A_1V\\
 A_0A_1^2A_0&= A_0A_1A^2V
 \end{array}\right.
 $$
 Since $A_0$ and $A_1$ are invertible, we get $A_0=A_1V$ and $A_1A_0= A^2V$. \ We deduce in particular that $V$ is invertible and that $A_1^2V= A^2V$.  We obtain $A_1^2= A^2$ and finally $A_1=A$, as required.\\

 We give another proof, of independent interest, of the forward propagation property of subnormal operator weighted shifts. \ It is based on the following lemma on semi-spectral measures.
\begin{lem}\cite[Theorem 1.1]{pietrzycki2022two}\label{stoshelmn}
A Borel semi-spectral measure $E$ on $\mathbb{R}$ with compact support is spectral if and only if
\(
\displaystyle\left(\int_{\mathbb{R}} tdE(t)\right)^2 = \int_{\mathbb{R}} t^2dE(t).
\)
\end{lem}
\begin{proof} Let $W_\mathcal{A}$ be a subnormal operator weighted shift such that $A_p=A_{p+1}$. \ By Theorem \ref{sub}, there exists an operator-valued measure  $E$ supported on $K=[0,\|W_{\mathcal A}\|^2]$, such that:
\[
B_{n}^*B_n = \int_K t^n dE(t) \text{ for every } n\in\mathbb{Z}_+.
\]
We have:
\begin{equation}\label{a1}
   A_{p}^2 = (B_{p}^{*})^{-1}(B_{p+1}^*B_{p+1})B_{p}^{-1}  = \int_K t dE_p(t),
\end{equation}
where $dE_p(t) = t^p(B_{p}^{*})^{-1}dE(t)B_{p}^{-1}$, which is a semi-spectral measure, since $E_p(\mathbb{R})=I_{\mathcal{H}}$. \
Now,  $A_p = A_{p+1}$ implies:
\begin{equation}\label{b1}
   A_{p}^4 = (B_{p}^{*})^{-1}(B_{p+2}^*B_{p+2})B_{p}^{-1} = \int_K t^2 dE_p(t).
\end{equation}
From equations \eqref{a1} and \eqref{b1} we conclude:
\begin{equation*}
     \left(\int_K tdE_p(t)\right)^2 = \int_K t^2dE_p(t).
\end{equation*}
So according to Lemma \ref{stoshelmn}, $E_p$ is a spectral measure. \ In particular, the following identity holds:
\[
\left(\int_K tdE_p(t)\right)^n = \int_K t^ndE_p(t) \text{ for every } n\ge 3.
\]
This leads to the relationship:
\begin{equation}\label{5}
   B_{n+p}^*B_{n+p} = B_{p}^*A_{p}^{2n}B_{p} \text{ for every } n \ge 3.
\end{equation}
We can now proceed as follows:
\begin{itemize}
   \item let $n=3$ in \eqref{5}; then, $A_{p+2} = A_p$;
   \item let $n=4$ in \eqref{5}; then, $A_{p+3} = A_p$.
   \item By induction on $n$, we obtain $A_n = A_p$ for every $n \ge p.$
\end{itemize}
\end{proof}
\section{ Tchakaloff's Theorem and Operator Moment Problems} Tchakaloff's Theorem for scalar truncated moment sequences states that if a measure admits finite moments $s_k$ for every $k=0,\dots, r$, then $(s_0,\dots,s_r)$ admits a finitely atomic representing measure \cite{tcha}. \ Various authors have used this result to show that the full scalar moment problem and the truncated scalar moment problem are equivalent, initiated by J. Stochel in  \cite{Sto}. \ Tchakaloff's Theorem has recently been extended to the matrix case in \cite{madler2023truncated}.  \\
In contrast with the finite-dimensional case (scalar or matricial), Tchakaloff's Theorem is not valid in the infinite-dimensional case. \ More precisely, D.P. Kimsey has provided an example of initial data $ (T_0, T_1)$ with a representing operator-valued measure but with no finitely atomic representing measure; see \cite{kims}. \  It should be noted, however, that in Kimsey's example, the representing measure of  $ (T_0, T_1)$ has no bounded  moments for $n\ge 3$. \ Thus is not $ (T_0, T_1)$ necessarily the initial data of an operator moment sequence. \ For completeness, below we include a variation of Kimsey's example, using different arguments, in which  $ (T_0, T_1)$ is the initial data of an operator moment sequence. 
\begin{ex}\cite[Example 1]{kims}\label{ex71} Let $E=\text{diag}(e^{-n}\delta_{-n})_{n=1}^{+\infty}$ be an OVM, and let $T_0=\text{diag}(e^{-n})_{n=1}^{+\infty}$ and $T_1=-\text{diag}(ne^{-n})_{n=1}^{+\infty}$ be the first two moments of $E$ defined on the Hilbert space $\ell^2(\mathbb{Z}_+)$ of square-summable sequences. 
\end{ex} Since $T_k=\text{diag}((-n)^ke^{-n})_{n=1}^{+\infty}$ is a bounded self-adjoint operator on $\ell^2(\mathbb{Z}_+)$ for every $k\in\mathbb{Z}_+$, we get $(T_0, T_1)$ is a truncated operator moment sequence.\\
We will show that $(T_0, T_1)$ has no finitely atomic representing measure. \ Otherwise, if $F=\sum\limits_{k=1}^r P_k\delta_{\lambda_k}$ for some $\lambda_1<\cdots<\lambda_r$, then  $T_0=\sum\limits_{k=1}^r P_k$ and $T_1=\sum\limits_{k=1}^r \lambda_k P_k$. \ In particular, for $e_n$, the $n^{th}$ vector in the canonical basis of $\ell^2(\mathbb{Z}_+)$, we would get
\[
\left\{
\begin{array}{lll}
    \langle T_0e_n, e_n\rangle &=& \sum\limits_{k=1}^r \langle P_ke_n, e_n\rangle = e^{-n} \\
    \langle T_1e_n, e_n\rangle &=& \sum\limits_{k=1}^r \lambda_k \langle P_ke_n, e_n\rangle = -ne^{-n}
\end{array}
\right.
\]
It follows that
\[
\sum\limits_{k=1}^r (n+\lambda_k) \langle P_ke_n, e_n\rangle = 0 \text{ for every } n\in\mathbb{Z}_+.
\]
The last fact is impossible for $n+\lambda_1>0$, since $\langle P_ke_n, e_n \rangle \geq 0$ for every $k=1,\dots,r$ and $\sum\limits_{k=1}^r \langle P_ke_n, e_n\rangle = e^{-n} \ne 0$.\\
~\\
    It is then natural to ask under what conditions a version of Tchakaloff's Theorem can be obtained in our setting. \ In contrast with Example \ref{ex71}, we have: 
  \begin{prop}\label{finite-atomic} Let $T_0, T_1\in {\mathcal B}_h({\mathcal H})$. \ The following statements are equivalent:
  \begin{enumerate}
   \item There exists $\alpha,\beta \in {\mathbb R}$ $(\alpha\neq\beta)$ such that $$\alpha \langle T_0x,x\rangle\le \langle T_1x,x\rangle\le \beta\langle T_0x,x\rangle \ \mbox{ for every }\ x\in {\mathcal H};$$
   \item  $(T_0,T_1)$ admits a $2$--atomic operator-valued representing measure;
  \item  $(T_0,T_1)$ admits a finitely atomic operator-valued representing measure;
    \item  $(T_0,T_1)$ admits a compactly supported operator-valued representing measure.
  \end{enumerate}
  \end{prop}
  \begin{proof} ~
       \begin{itemize}
  \item $1\Rightarrow 2$. \ By solving   the operator equations $P_1+P_2=T_0$ and  $\alpha P_1+\beta P_2=T_1$, we obtain
  $$
  P_1=\frac{1}{\beta-\alpha}(\beta T_0-T_1) \ge 0 \textrm{ and } P_2=\frac{1}{\beta -\alpha}(T_1-\alpha T_0)\ge 0.
  $$
  It follows that $E:=P_1\delta_{\alpha} +P_2\delta_{\beta}$ is a $2$--atomic OVM representing $(T_0,T_1)$. 
  \item $2\Rightarrow 3$ and $3\Rightarrow 4$ are clear.
  \item $4\Rightarrow 1$. \ Suppose that $(T_0,T_1)$ admits an operator-valued representing measure $E$ supported in a compact set $K$. \ Denote  $\alpha= \min(K)$ and $\beta= \max(K)$. \ We have 
  $$
  \langle T_0x,x\rangle =\int_Kd\langle E(t)x,x\rangle \textrm{ and } \langle T_1x,x\rangle=\int_Ktd\langle E(t)x,x\rangle,
  $$
  for every $ x\in {\mathcal H}$. \ Moreover,
  $$\alpha \int_Kd\langle E(t)x,x\rangle\le \int_Ktd\langle E(t)x,x\rangle \le \beta\int_Kd\langle E(t)x,x\rangle.
   $$
  Thus 1 holds.
  \end{itemize}
  \end{proof}
  As an immediate consequence, we deduce.
  \begin{prop} Let $T\in {\mathcal B}_h({\mathcal H})$. \ We have
  \begin{enumerate}
  \item  $(I,T)$ is a truncated moment sequence associated with a $2$--atomic operator-valued representing measure, 
  \item For  every positive and invertible operator $S$,  the pair $(S,T)$ is a truncated moment sequence associated with a $2$--atomic operator-valued representing measure.
  \end{enumerate}
  \end{prop}  
   Given the preceding discussion, the next challenging problem arises naturally. \begin{problem}\label{pbfinite}
        Let $(T_0,\dots, T_r)\in\mathbf{B}_h(\mathcal{H})$ $(r\ge 3)$.  \ Under what conditions are the following statements equivalent?
    \begin{enumerate}
    \item $(T_0,\dots, T_r)$ admits an operator-valued  representing  measure;
      \item $(T_0,\dots, T_r)$ admits an $r$--atomic operator-valued representing  measure.
\end{enumerate}  
\end{problem}
   In connection with the local operator moment problem, we have the next relation. \ If  $(T_0,\cdots, T_r)$ admits an operator-valued  representing  measure, then $(\langle T_0x,x\rangle,\dots,\langle T_rx,x\rangle)$ admits a finitely atomic operator-valued representing  measure. \ The converse does not hold in general.  \ Indeed,  for any given numbers $s_0>0$ and $s_1\in {\mathbb R}$. \ The pair $(s_0,s_1)$ is a truncated moment sequence admitting as representing $1$--atomic measure $\mu=s_0\delta_{r}$, with $r=\frac{s_1}{s_0}$. \ Taking in count this last fact and  Example \ref{ex71}, we deduce that
    \begin{prop}\label{finite-atomic2} 
There exists a truncated operator moment sequence \((T_0, T_1)\) such that: 
\begin{enumerate}
    \item \((T_0, T_1)\) admits no finitely atomic operator-valued representing measure;
    \item For every \(x \in \mathcal{H}\), \((\langle T_0x, x\rangle, \langle T_1x, x\rangle)\) admits a 1-atomic operator-valued representing measure.
\end{enumerate}    
\end{prop}
\begin{rem}\begin{enumerate}
    \item 
 We mention at this stage that Proposition \ref{finite-atomic2} contradicts \cite[Theorem 3]{kims}. \ In fact, in the proof of \cite[Theorem 3]{kims}, a determinacy assumption on local measures representing the local moment sequences is needed. \ We also notice that in Example \ref{ex71}, Condition (1) in Proposition \ref{finite-atomic2} is not satisfied.
 \item The first assumption in Proposition \ref{finite-atomic2} shows that the equivalence in  Problem \ref{pbfinite} is not possible without an additional assumption. \ Because of Proposition \ref{finite-atomic2}, we think that a  reasonable requirement might be the following: $(T_0,\dots, T_r)$ admits a  compactly supported operator-valued representing measure.
 \end{enumerate}
\end{rem}
Using recursive scalar-moment sequences, we recover \cite[Theorem 4]{kims}. \ More precisely, we have:
\begin{cor} (\cite[Theorem 4]{kims}) \label{kimm} \ Let $T_0,\dots,T_r\in\mathbf{B}_h(\mathcal{H})$,  with an operator-valued measure $E$ and let $\Sigma$ be a finite set. \ The following statements are equivalent:
\begin{enumerate}
    \item There exist finitely atomic measures $(\mu_x)_{x\in {\mathcal H}}$  such that $\mu_x$ is a representing  measure for the truncated sequence $(\langle T_0x,x\rangle,\dots,\langle T_rx,x\rangle)$ and supported in $\Sigma$ for every $x\in\mathcal{H}$;
    \item $supp(E)\subseteq \Sigma$.
\end{enumerate}
\end{cor}
\begin{proof}
 Since $supp(\mu_x)\subseteq supp(E)\subseteq \Sigma$, only the direct implication requires a proof. \ To this aim, for every $x\in\mathcal{H}$, $\langle \mathcal{T}x,x\rangle$ is a {\it recursive sequence} associated with the polynomial $P(X) =\displaystyle \prod\limits_{\lambda\in\Sigma}(X-\lambda)=X^p-\sum\limits_{k=1}^{p-1}a_kX^k$, where $p=card(\Sigma)$ and $a_k$ real numbers. \ Thus,
$ \langle T_{n}x,x\rangle=\displaystyle\sum\limits_{k=1}^{p-1}a_k\langle T_{n-k}x,x\rangle$ for every  $x\in\mathcal{H}.$ 
\ By  the polarization formula, we obtain  $  \langle T_{n}x,y\rangle=\displaystyle\sum\limits_{k=0}^{p-1}a_k\langle T_{n-k}x,y\rangle$   for every  $  x,y\in\mathcal{H}.$ \ It follows that 
$  T_{n}=\displaystyle\sum\limits_{k=1}^{p-1}a_kT_{n-k}$ and then that $supp(E)\subseteq \Sigma$.
\end{proof}

   The previous discussion and results suggest the next general challenging problem.
    \begin{problem}\label{58} For  $T_0,\dots,T_r\in\mathbf{B}_h(\mathcal{H}) $ with an operator-valued representing measure, under what conditions are the two following statements equivalent?
\begin{enumerate}
    \item $(\langle T_0x,x\rangle,\dots,\langle T_rx,x\rangle)$ admits a finitely atomic   representing  measure  for every nonzero $x\in {\mathcal H}$;
      \item $(T_0,\dots,T_r)$ admits a finitely atomic operator-valued representing  measure.
\end{enumerate}
 \end{problem} 
 \begin{rem} 
It is worth noting that in Problems \ref{pbfinite} and \ref{58}, only \( 1. \Rightarrow 2. \) requires a proof. \ Also, an affirmative answer to Problem \ref{58} follows from an affirmative answer to Problem \ref{pbfinite}. \ Indeed, assume that Problem \ref{58} holds true. \ If the statement 1 of Problem \ref{pbfinite} is given, i.e., \( (T_0, \dots, T_r) \) admits an operator-valued representing measure, then by the scalar version of Tchakaloff's theorem, \( (\langle T_0x, x\rangle, \dots, \langle T_rx, x\rangle) \) admits a finitely atomic representing measure for every nonzero \( x \in \mathcal{H} \). \ By \( 1. \Rightarrow 2. \) in Problem \ref{58}, we derive  that \( (T_0, \dots, T_r) \) admits a finitely atomic operator-valued representing measure.
\end{rem}

A reasonable condition would be $(T_0,\dots,T_r)$ admits a compactly supported operator-valued representing measure. \ This statement is not true, as shown by the next example.
\begin{ex}
    Let $(a_n)_{n\in\mathbb{Z}_+}\in l^2({\mathbb N})$ be a real non negative sequence, we associate to $(a_n)_{n\in {\mathbb N}}$ the self-adjoint diagonal operator $T=diag(a_n)_{n=1}^{+\infty}$. \ The sequence $(I,T,T^2,T^3,\dots)$ is clearly an operator $\sigma(T)$--moment sequence. Thus $(\langle x,x\rangle,\langle Tx,x\rangle,\langle T^2x,x\rangle,\dots)$ is a scalar moment sequence for every $x\in {\mathcal H}$. \ In particular  $(\langle x,x\rangle,\langle Tx,x\rangle,\langle T^2x,x\rangle)$ admits a finitely atomic representing measure for   every $x\in {\mathcal H}$.
    
    Let us show that  $(I,T,T^2)$ has no finitely atomic representing measure. \ Otherwise there is $r\ge 1$, $\lambda_1<\dots<\lambda_r$ and $P_1,\dots,P_r \in \mathbf{B}_+(l^2({\mathbb N}))$ such that
    \[
    \left\{ \begin{array}{l}
        I= \displaystyle\sum_{k=1}^r P_k  \\ T= \displaystyle\sum_{k=1}^r \lambda_k P_k \\ T^2=\displaystyle\sum_{k=1}^r \lambda_k^2 P_k
    \end{array}\right.\; \Longrightarrow\quad  \left\{ \begin{array}{l}
        1= \displaystyle\sum_{k=1}^r \langle P_ke_n,e_n\rangle  \\ a_n= \displaystyle\sum_{k=1}^r \lambda_k \langle P_k e_n,e_n\rangle\\ a_n^2=\displaystyle\sum_{k=1}^r \lambda_k^2 \langle P_k e_n,e_n\rangle
    \end{array}\right. \quad n\ge 1
    \]
    In particular,  for every $n\ge 1$, we have
    \[
    a_n=\displaystyle\sum_{k=1}^r \lambda_k \langle P_k e_n,e_n\rangle=\left( \displaystyle\sum_{k=1}^r \lambda_k^2 \langle P_k e_n,e_n\rangle\right)^{1/2}\left( \displaystyle\sum_{k=1}^r \langle P_k e_n,e_n\rangle\right)^{1/2}
    \]
    Thus, equality holds in the Cauchy-Schwartz inequality, and then  there is a real number $\alpha$ such that:
    \begin{equation}\label{eq111}
    \left( \begin{matrix}
        \lambda_1 {\langle P_1 e_n,e_n\rangle}^\frac{1}{2}\\ \vdots \\ \lambda_r {\langle P_r e_n,e_n\rangle}^\frac{1}{2}
    \end{matrix}\right)=\alpha  \left( \begin{matrix}
 {\langle P_1 e_n,e_n\rangle}^\frac{1}{2}\\ \vdots \\   {\langle P_r e_n,e_n\rangle}^\frac{1}{2}
    \end{matrix}\right) \   \mbox{ for arbitray }  n\ge 1.
    \end{equation}
    Since $\lambda_i\ne \lambda_j$  for  $i\ne j$, and   $1= \displaystyle\sum_{k=1}^r \langle P_ke_n,e_n\rangle$,  it follows from Equation \ref{eq111} that there is a unique $k_n\in 
    \{1, \cdots,r\}$ such that $\langle P_{k_n} e_n,e_n\rangle\ne 0$. \ Clearly we have $\langle P_{k_n} e_n,e_n\rangle=1$ and because $P_{k_n}$ is a contraction, we deduce that $P_{k_n}e_n=e_n$. \ Now, using the identity $T= \displaystyle\sum_{k=1}^r \lambda_k P_k $, we readily obtain $Te_n=\lambda_{k_n}e_n$, and from this fact it follows  that $a_n=\lambda_{k_n}$. \\
    For $k=1, \cdots r$, denote  $E_k= \{ n \in \mathbb N : k_n=k \}$. \ We have $\displaystyle\bigcup_{k\in\mathbb{N}} E_k=\mathbb N$, and then there is at least one $k \in \{1, \cdots,r\}$ such that $E_k$ is infinite. \ In particular $a_n=\lambda_{k}\ne 0$ on an infinite set, which is impossible, since $(a_n)_{n\in{\mathbb N}}\in l^2({\mathbb N})$.
    \end{ex}
    \begin{rem} From the previous computations, we derive  the following useful remarks:
        \begin{enumerate}
            \item $(I,T,T^2)$ has a finitely atomic representing measure if and only is, there exists $k\in {\mathbb N}$ such that $a_n=0 $ for every $n\ge k$. \ Equivalently, $T$ is a finite rank operator.
            \item The number of atoms of the finitely atomic representing measure ( when exists ) is equal to the cardinal of the set $\{a_n : n\in {\mathbb N}\}.$ \ Thus, in contrast with the scalar case, and the matrix case, the minimal rank of representing finite atomic measures of operator moment extensions of three initial operator data is not bounded. 
        \end{enumerate}
    \end{rem}

\section{Linear Recursive Sequences  and Finitely Atomic OVMs}
In this section, we will give necessary and sufficient conditions for the recursive operator moment problem to have a solution on a finite subset of $\mathbb{R}$, described in terms of recursiveness.\\ 
We consider first the algebra endomorphism $\tau$ on $\mathbf{B}_h(\mathcal{H})^{\mathbb{Z}_+}$  defined as follows:$$\begin{array}{ccccc}
 \tau & : & \mathbf{B}_h(\mathcal{H}) ^{\mathbb{Z}_+} & \to & \mathbf{B}_h(\mathcal{H}) ^{\mathbb{Z}_+} \\
 & & (T_n)_{n\in\mathbb{Z}_+ }=(T_0,T_1, T_2,\cdots )& \mapsto & (T_{n+1})_{n\in\mathbb{Z}_+}=(T_1, T_2,T_3,\cdots ).\\
\end{array}$$ For $k\in\mathbb{Z}_+$,  its $k^{th}$ iteration  $\tau^k$ on $\mathbf{B}_h(\mathcal{H}) ^{\mathbb{Z}_+}$ is given by:
$$\begin{array}{ccccc}
\tau^k & : &  \mathcal{B}(\mathcal{H})_{h}^{\mathbb{Z}_+} & \to &  \mathcal{B}(\mathcal{H})_{h}^{\mathbb{Z}_+} \\
 & & (T_n)_{n\in\mathbb{Z}_+} & \mapsto & (T_{n+k})_{n\in\mathbb{Z}_+}, \\
\end{array}$$
and for  $P(X) = \displaystyle\sum_{k=0}^{r} a_k X^k\in\mathbb{R}[X]$,  $P(\tau)$ on $\mathbf{B}_h(\mathcal{H}) ^{\mathbb{Z}_+}$  is the following:
$$\begin{array}{ccccc}
P(\tau) & : &\mathcal{B}(\mathcal{H})_{h}^{\mathbb{Z}_+} & \to & \mathbf{B}_h(\mathcal{H}) ^{\mathbb{Z}_+} \\
 & & (T_n)_{n\in\mathbb{Z}_+} & \mapsto & (a_0T_{n}+a_1T_{n+1}+\dots+a_rT_{n+r})_{n\in\mathbb{Z}_+}. \\
\end{array}$$
\begin{defn}
A sequence ${\mathcal T}=(T_n)_{n\in\mathbb{Z}_+}\in\mathbf{B}_h(\mathcal{H}) ^{\mathbb{Z}_+}$ is called a {\it linear recursive sequence} (abbreviated {\rm LRS}) if there exists a non-zero polynomial $P\in \mathbb{R}[X]$ such that ${\mathcal T}\in \ker P(\tau)$, where $\ker P(\tau)$ represents the null space of the algebra endomorphism $P(\tau)$. \ In this case, $P$ is called a {\it characteristic polynomial} associated with the sequence ${\mathcal T}$.
\end{defn}
By convention, the zero polynomial is considered a characteristic polynomial for every operator sequence. \ Let  ${\mathcal T}$ be a linear recursive sequence  and consider the following algebra homomorphism:
$$\begin{array}{ccccc}
\Psi_{{\mathcal T}}(\tau) & : & \mathbb{R}[X] & \to & \mathbf{B}_h(\mathcal{H}) ^{\mathbb{Z}_+} \\
 & & P & \mapsto & P(\tau)({\mathcal T}). \\
\end{array}$$
We denote by $\mathcal{P}(\mathcal{T}) = \ker \Psi_{\mathcal{T}}(\tau) \subseteq \mathbb{R}[X]$ the set of all associated characteristic polynomials with $\mathcal{T}$. \ It is clear that $\mathcal{P}({\mathcal T})$ is an ideal of the principal ring $\mathbb{R}[X]$. \ Thus,  there exists a unique monic polynomial $P_{\mathcal T}\in\mathcal{P}({\mathcal T})$ of minimal degree, such that  $\mathcal{P}({\mathcal T})=P_{\mathcal T}\mathbb{R}[X].$ \
 We will say that
 \begin{itemize}
     \item $P_{\mathcal T}$ is the {\it minimal polynomial} of  ${\mathcal T}$;
     \item the relation $P_{\mathcal T}(\tau)({\mathcal T})=(0_\mathcal{H})_{n\in\mathbb N}$ is the {\it minimal linear recurrence relation} of ${\mathcal T}$; and 
     \item the {\it order} of the recursive sequence, is $r:=\text{deg}(P_{\mathcal T})$.
 \end{itemize}

It is immediate, that if ${\mathcal T}$ is an LRS with the minimal polynomial $P_{{\mathcal T}}$, then for every nonzero $x\in \mathcal{H}$, the sequence $\langle{\mathcal T}x,x\rangle$  is also an LRS  with the minimal polynomial $P_{\langle{\mathcal T}x,x\rangle}$ such that
\begin{equation}\label{mu}
  P_{\mathcal T} \mbox{ is a multiple of } P_{\langle{\mathcal T}x,x\rangle}.   
\end{equation}
Noticing that for every $x \in \mathcal{H}$, the set $\left\{P \in \mathbb{R}[X] \mid P(\tau)(\langle{\mathcal T}x,x\rangle) = (0)_{n\in\mathbb{Z}_+}\right\}$ is an ideal generated by $P_{\langle{\mathcal T}x,x\rangle}$, we derive that
\begin{equation}\label{lcm}
P_{\mathcal T}\mathbb{R}[X]=\displaystyle\bigcap_{x\in\mathcal{H}}\left\{P\in\mathbb{R}[X]  \mid P(\tau)(\langle{\mathcal T}x,x\rangle)=(0)_{n\in\mathbb{Z}_+}\right\}=\displaystyle\bigcap_{x\in\mathcal{H}}P_{\langle{\mathcal T}x,x\rangle}\mathbb{R}[X].
\end{equation}

From \eqref{mu} and \eqref{lcm}, we deduce the following useful remark. 
\begin{rem}\label{fin}
Since the set  $\left\{P_{\langle{\mathcal T}x,x\rangle}  \mid  x \in\mathcal{H}\right\}$ is finite, there are $x_1,\cdots,x_p\in\mathcal{H}$ such that 
\begin{equation}\label{finite}
P_{\mathcal T}\mathbb{R}[X]=\displaystyle\bigcap_{1\le k\le p}P_{\langle{\mathcal T}x_k,x_k\rangle}\mathbb{R}[X].
\end{equation}
\end{rem}

Next, we describe the minimal polynomial $P_{{\mathcal T}}$ in terms of the minimal polynomials $P_{\langle{\mathcal T}x_k,x_k\rangle}$ for $k=1,\dots,p$. 
\begin{prop}
Under the above notation, we have
$$P_{{\mathcal T}}=\operatorname{l.c.m.}\left\{P_{\langle{\mathcal T}x_k,x_k\rangle}  \mid 1\le k\le p\right\},$$ where $\operatorname{l.c.m.}$ denotes the least common multiple.
\end{prop}
\begin{proof}
We have 
$$P_{\mathcal T}\mathbb{R}[X] \overset{\eqref{finite}}{=}\displaystyle\bigcap_{1\le k\le p}P_{\langle{\mathcal T}x_k,x_k\rangle}\mathbb{R}[X]=\operatorname{l.c.m.}\left\{P_{\langle{\mathcal T}x_k,x_k\rangle}  \mid 1\le k\le p\right\}\mathbb{R}[X].$$ By the uniqueness of $P_{\mathcal T}$, we conclude that $P_{\mathcal T}=\operatorname{l.c.m.}\left\{P_{\langle{\mathcal T}x_k,x_k\rangle}  \mid 1\le k\le p\right\}.$
 \end{proof}

The next proposition is straightforward.
 \begin{prop}\label{aaa}
With the above notation, we have
\begin{enumerate}
    \item   $\mathcal{Z}(P_{{\mathcal T}})=\displaystyle\bigcup_{k=1}^{p}\mathcal{Z}(P_{\langle{\mathcal T}x_k,x_k\rangle})$.
    \item $P_{{\mathcal T}}$ has only simple roots if and only if $P_{\langle{\mathcal T}x_k,x_k\rangle}$ has only simple roots for every $k=1,\dots,p$.
\end{enumerate}         
\end{prop}

In the sequel, let ${\mathcal T} = (T_n)_{n\in\mathbb{Z}_+}\in \mathbf{B}_h(\mathcal{H}) ^{\mathbb{Z}_+}$ be an LRS with  minimal polynomial $P_{\mathcal T}$ of degree $r$. \ For every nonzero $x\in{\mathcal{H}}$ and $n\in \mathbb{Z}_+$, we associate the local infinite and finite-type Hankel matrices, which are respectively defined as follows:
$$
H(x) = \big(\langle T_{i+j}x,x\rangle_\mathcal{H}\big)_{i,j\in \mathbb{Z}_+} \mbox{ and } H_n(x) = \big(\langle T_{i+j}x,x\rangle_\mathcal{H}\big)_{0\le i,j \le n}.
$$

By identifying the polynomial $P(X)=\displaystyle\sum_{k=0}^{n}a_kX^k$ with the column vector $\widehat{P}=(a_0 ,\dots, a_n , 0,0,  \dots  )^{t}$, we derive the following trivial lemma.
\begin{lem}\label{car} Using the above notation, we have
    $P\in \mathcal{P}({\mathcal T}) \iff  H(x)\widehat{P}=\widehat{0}$ for every $x\in {\mathcal{H}}$. \
    In particular, $H(x)\widehat{P_{\mathcal T}}=\widehat{0}$ for every $x\in {\mathcal{H}}$.
   \end{lem} 

We also have the following conditions for positivity:
$$\begin{array}{ccl}
     H(x)\ge 0  & \iff&  H_n(x)\ge 0 \mbox{ for every } n\in \mathbb{Z}_+,\\
     &\iff& \widehat{P}^{t}H(x)\widehat{P} \ge 0 \mbox{ for every }  P\in\mathbb{R}[X].
\end{array} 
$$  
 
We derive the following crucial proposition.
\begin{prop}\label{eqpo} Let   ${\mathcal T}$ be a {\rm  LRS} of order $r$ and $x\in{\mathcal{H}}$ be  a nonzero vector. \ The following statements are equivalent.
\begin{enumerate}
    \item  $H(x)\ge 0$;
    \item $H_{r-1}(x)\ge 0$.
\end{enumerate}
\end{prop}
\begin{proof}
The direct implication is clear. \ To show the converse, let $x\ne 0$, be such that $H_{r-1}(x)\ge 0$. \ For every  $P\in\mathbb{R}[X]$, we use the Euclidean division algorithm to write $P$ as $P=QP_{\mathcal T}+R$ with $\text{deg}(R)\leq r-1$. \ It follows that  $$\widehat{P}^{t}H(x)\widehat{P}=\widehat{R}^{t}H_{r-1}(x)\widehat{R}\geq 0.$$  Finally, $H(x)\ge 0$.
\end{proof}

We say that a polynomial $P\in\mathbb{R}[X]$ has distinct  roots if ${\mathcal Z}(P)=\big\{\lambda_{1},\lambda_{2},\dots,\lambda_{r}\big\}$, with only simple roots. \ The following lemma provides the condition for $P_{\mathcal T}$ to have distinct roots.
\begin{lem}\label{dis}  Let   ${\mathcal T}$ be an {\rm  LRS} of order $r$ such that  $H_{r-1}(x)\ge 0$  for every $x\ne 0$. \ Then its minimal polynomial $P_{\mathcal T}$ has distinct roots.
\end{lem}
\begin{proof}
If $H_{r-1}(x)\ge 0$  for every $x\ne 0$ , then by Proposition \ref{eqpo}, we have  $H(x)\ge 0$. \ We deduce that $\langle{\mathcal T}x,x\rangle$ is a scalar moment sequence for  every nonzero $x\in {\mathcal{H}}$. \ Using \cite{taher2001recursive}, we derive that $P_{\langle{\mathcal T}x,x\rangle}$ has only simple roots for  every nonzero $x\in {\mathcal{H}}$. \ By appealing to Proposition \ref{aaa}, we conclude that the minimal polynomial $P_{\mathcal T}$ has distinct roots.
\end{proof}

The following lemma specifies the expression of the sequence ${\mathcal T}$ when the minimal polynomial has simple roots. 
\begin{lem}\label{ge} Let  ${\mathcal T}$ be an {\rm  LRS } and let $P_{\mathcal T}$ be its minimal polynomial  with $\mathcal{Z}(P_{\mathcal T})=\big\{\lambda_{1},\lambda_{2},...,\lambda_{r}\big\}$. \ The following statements are equivalent:
\begin{enumerate}
    \item $\mathcal{Z}(P_{\mathcal T})$ consists only with simple roots; 
    \item ${\mathcal T}$ admits a finitely atomic representing {\rm OVC}  supported in $\mathcal{Z}(P_{\mathcal T})$;
    \item  For every nonzero $x\in \mathcal{H},$ $\langle{\mathcal T}x,x\rangle$ admits a finitely atomic scalar representing charge  supported in $\mathcal{Z}(P_{\mathcal T})$;
    \item $\langle{\mathcal T}x_k,x_k\rangle$ admits a finitely atomic scalar representing charge supported in $\mathcal{Z}(P_{\mathcal T})$ for   $k=1,\dots,p$. \ Here  $x_1,\cdots,x_p\in\mathcal{H}$ are given by Remark \ref{fin}, and such that 
$P_{\mathcal T}\mathbb{R}[X]=\displaystyle\bigcap_{1\le k\le p}P_{\langle{\mathcal T}x_k,x_k\rangle}\mathbb{R}[X]$.
\end{enumerate}
Moreover, the finitely atomic representing {\rm OVC} is given by: $ E=\displaystyle\sum_{k=1}^{r}S_k\delta_{\lambda_k}$, where $S_1,\dots, S_r \in\mathbf{B}_h(\mathcal{H}) $.
\end{lem}
\begin{proof} Clearly $2 \Rightarrow  3  \Rightarrow  4$ holds and from Proposition \ref{aaa} and \cite[Proposition 2.4]{taher2001recursive}, we have $1 \Leftrightarrow  4.$  \  It remains then to show    $1  \Rightarrow  2$.  \ Indeed, if $P_{{\mathcal T}}(X)=\displaystyle\prod_{k=1}^{r}(X-\lambda_k)$, then according to the kernel lemma  with   $\tau^0=id$, we have
\begin{equation}\label{kl}
\ker P_{{\mathcal T}}(\tau)=\ker(\tau-\lambda_1id)\oplus \ker(\tau-\lambda_2id)\oplus...\oplus \ker(\tau-\lambda_rid).   
\end{equation}
Since, ${\mathcal T}=(T_n)_{n\in\mathbb{Z}_+}\in\ker P_{{\mathcal T}}(\tau)$, from equation \eqref{kl}, there exists a unique ${\mathcal T}^{(k)}=(T_{n}^{(k)})_{n\in\mathbb{Z}_+}\in \ker(\tau-\lambda_kid) \; \; (k=1,...,r)$ such that 
\begin{equation}
     {\mathcal T}={\mathcal T}^{(1)}\oplus...\oplus {\mathcal T}^{(r)}\Leftrightarrow T_n=T_{n}^{(1)}+...+T_{n}^{(r)} \mbox{ for all } n\in\mathbb{Z}_+.
\end{equation}
For $k=1,...,r$, we have:
$$\begin{array}{ccccc}
{\mathcal T}^{(k)}\in \ker(\tau-\lambda_k \mathrm{id}) & \Leftrightarrow &(T_{n+1}^{(k)})_{n\in\mathbb{Z}_+}&=&(\lambda_kT_{n}^{(k)})_{n\in\mathbb{Z}_+}, \\
                                         & \Leftrightarrow & T_{n+1}^{(k)}&=&\lambda_kT_{n}^{(k)} \mbox{ for all } n\in\mathbb{Z}_+, \\
                                         & \Leftrightarrow & T_{n}^{(k)}&=&\lambda_{k}^{n}T_{0}^{(k)} \mbox{ for all } n\in\mathbb{N}.\\
\end{array}$$
For $k=1,\dots, r$, pose $S_k=T_{0}^{(k)}\in\mathbf{B}_h(\mathcal{H}) $. \ Then, we have $  T_n=\lambda_{1}^{n}S_1+\lambda_{2}^{n}S_2+\dots+\lambda_{r}^{n}S_r$ for every $n\in\mathbb{Z}_+$.
\end{proof}
\begin{rem}
   We obtain $2 \Rightarrow 1$ in the following way. \ Assume that there exists  $S_1,\dots, S_r\in \mathbf{B}_h(\mathcal{H})$ such that for all $n\in\mathbb{Z}_+$, 
$  T_n=\lambda_{1}^{n}S_1+\lambda_{2}^{n}S_2+\dots+\lambda_{r}^{n}S_r.
$
Let $P(X)=\displaystyle\prod_{k=1}^{r} (X-\lambda_k)=\sum_{k=0}^r a_k X^k.$ We have,
 \begin{align*}
    P(\tau)({\mathcal T})& =\left( \sum_{k=0}^r a_k T_{n+k}\right)_{n\in\mathbb N}\\
    &= \left( \sum_{k=0}^r a_k( \lambda_{1}^{n+k}S_1+\dots+ \lambda_{r}^{n+k}S_r)\right)_{n\in\mathbb N}\\
    &= \left( P(\lambda_1)\lambda_1^nS_1+\dots +P(\lambda_r)\lambda_r^nS_r\right)_{n\in\mathbb N}\\
    &=(0_\mathcal{H})_{n\in\mathbb N}.
\end{align*}
In particular, $P\in \mathcal{P}({\mathcal T})$, which implies that the polynomial $P_{\mathcal T}$ divides $P$, and therefore  $P_{\mathcal T}$ has simple distinct roots.
\end{rem}

Using the proof of Proposition 3.8 in \cite{curto1996solution} as a guide, we obtain the following extension of the scalar case.
\begin{thm}\label{mr}
Let ${\mathcal T}=(T_n)_{n\in\mathbb{Z}_+}\in\mathbf{B}_h(\mathcal{H}) ^{\mathbb{Z}_+}$ be an {\rm LRS} with minimal polynomial $P_{{\mathcal T}}$ of degree $r$. \ Then the following statements are equivalent:
\begin{enumerate}
     \item ${\mathcal T}$ is an operator $\mathcal{Z}(P_{{\mathcal T}})$--moment sequence;
     \item  ${\mathcal T}$ is a local operator $\mathcal{Z}(P_{{\mathcal T}})$--moment sequence;
    \item For every nonzero $x\in \mathcal{H}$, $H_{r-1}(x)\ge 0$.
\end{enumerate}
More precisely, if $\mathcal{Z}(P_{{\mathcal T}})=\big\{\lambda_{1},\lambda_{2},...,\lambda_{r}\big\}$, then the representing operator-valued measure associated with ${\mathcal T}$ admits the expression $E=\displaystyle\sum_{k=1}^{r}S_k\delta_{\lambda_k}$, where $S_1,\dots,S_r\in\mathbf{B}(\mathcal{H})_{+}$.
\end{thm}
\begin{proof}
Clearly $1\Rightarrow 2 \Rightarrow 3$. \ It remains then to show that $3\Rightarrow 1$. \ Indeed, suppose that $H_{r-1}(x)\ge 0$ for every $x\ne 0$. \ Then from Lemma \ref{dis}, we have $\mathcal{Z}(P_{{\mathcal T}})=\big\{\lambda_{1},\lambda_{2},...,\lambda_{r}\big\}$  consists of simple roots. \ According to Lemma \ref{ge}, there exists a {\it finitely atomic operator-valued charge}  $E=\displaystyle\sum_{k=1}^{r}S_k\delta_{\lambda_k}$ such that
$
T_n=\lambda_{1}^{n}S_1+\lambda_{2}^{n}S_2+...+\lambda_{r}^{n}S_{r} \mbox{ for all } n\in\mathbb{Z}_+.$ \
For $i=1,\cdots,r$, we will now prove that $S_i\ge 0$. \  Indeed,   using the Lagrange interpolation polynomial $\displaystyle L_{i}(X)=\prod_{\substack{j=0\\j\neq i}}^{r-1}\frac{(X-\lambda_{j})}{(\lambda_{i}-\lambda_{j})}$, \ we  obtain $\widehat{L_i}^{t}H_{r-1}(x)\widehat{L_i}=\langle S_ix,x\rangle_{\mathcal{H}}\ge 0$,  for every nonzero $x\in\mathcal{H}$. \ This concludes the proof.
\end{proof}
\begin{rem} Theorem \ref{mr} is equivalent Theorem (\cite[Theorem 4]{kims}) about Tchakaloff's Theorem in finite  dimension (see also  to Corollary \ref{kimm}, above). 
\end{rem}
\subsection{Two particular cases.}
\subsubsection{Recursive sequences of order  $r=1$  or $r=2$}
\begin{ex}$r=1.$
   Let ${\mathcal T}=(T_n)_{n\in\mathbb{Z}_+}\in\mathbf{B}_h(\mathcal{H}) ^{\mathbb{Z}_+}$, such that $T_{n+1}=\lambda T_n$ for every $n\in\mathbb{Z}_+$. \ Then,
 $  {\mathcal T} \text{  is an operator moment sequence
     if and only if } T_0\ge 0.$ \\
In this case, the associated representing {\rm OVM } is given by:  $E=T_0\delta_\lambda.$
\end{ex}
Using \cite[Proposition 4.1]{taher2001recursive} together with Theorem \ref{mr}, we derive the following result.
\begin{cor}$r=2.$\label{2a}
   Let ${\mathcal T}=(T_n)_{n\in\mathbb{Z}_+}\in\mathbf{B}_h(\mathcal{H}) ^{\mathbb{Z}_+}$ be an {\rm  LRS} satisfying  $T_{n+2}=(\lambda_1+\lambda_2) T_{n+1}-\lambda_1\lambda_2 T_n$  with $\lambda_1<\lambda_2$. \ Then the following statements are equivalent:
\begin{enumerate}
     \item ${\mathcal T}$ is an operator moment sequence;
     \item  $\langle T_0x,x\rangle\ge 0$ and $ \langle T_1x,x\rangle^2-(\lambda_1+\lambda_2)  \langle T_1x,x\rangle+\lambda_1\lambda_2  \le 0$; 
     \item for every $x\in \mathcal{H}$, $\begin{pmatrix}
   \langle T_0x,x\rangle& \langle T_1x,x\rangle \\
   \langle T_1x,x\rangle & \langle T_2x,x\rangle
\end{pmatrix}\ge 0$;
     \item $T_0\ge 0$, $T_2\ge 0$ and  for every $x\in \mathcal{H}$, we have $$\langle T_1x,x\rangle^{2}\le \langle T_2x,x\rangle\langle T_0x,x\rangle.$$
\end{enumerate}
Moreover, the associated representing {\rm OVM} is given by
$$
E=\displaystyle\frac{1}{\lambda_1-\lambda_2}(T_1-\lambda_2T_0) \delta_{\lambda_1} +  \frac{1}{\lambda_2-\lambda_1}(T_1-\lambda_1T_0)\delta_{\lambda_2}.
$$
\end{cor}
\
\begin{rem}
To see that $E$ is an OVM, it suffices to show  that $\lambda_1T_0\le T_1\le \lambda_2T_0$. \ Indeed, from 3 in the Corollary, we get for every $x\in {\mathcal{H}}$, $$\displaystyle(\langle T_1x,x\rangle-\lambda_1\langle T_0x,x\rangle)(\langle T_1x,x\rangle-\lambda_2\langle T_0x,x\rangle)=\langle T_1x,x\rangle^{2}-\langle T_2x,x\rangle\langle T_0x,x\rangle\le 0.$$
\end{rem}

\noindent {\bf Example \ref{ex71} revisited.} In Example \ref{ex71}, we have $T_0 \ge 0$, $T_2\ge 0$ and since $T_1^2=T_0T_2$, we obtain  $\langle T_1x,x\rangle^{2}\le \langle T_2x,x\rangle\langle T_0x,x\rangle, $ for every $x\in \mathcal{H}$. \  An {\rm  LRS} of order $2$ extending $(T_0,T_1)$ will satisfy $\lambda_1T_0\le T_1\le \lambda_2T_0$ and then $\lambda_1n^{-3}\le -n^{-2}\le \lambda_2n^{-3}$, for very $n\ge 1$. \ This yields $\lambda_1 \le -n \le \lambda_2$, which is impossible.\\

We end this section with an example of recursive operator moment sequences inspired by  \cite[Example 6.2]{pietrzycki2022two}.
\begin{ex}
     Consider the self-adjoint operator $B \in \mathbf{B}(\mathcal{H} \oplus \mathcal{H})$ given by the $2 \times 2$ block matrix:
$B =\begin{pmatrix}
    aI_{\mathcal{H}} & bI_{\mathcal{H}} \\
    bI_{\mathcal{H}} & cI_{\mathcal{H}}
\end{pmatrix}$ with $(a - c)^2+b^2\neq 0$. \ Let us also take $W : \mathcal{H} \to \mathcal{H} \oplus  \mathcal{H}$ defined by $Wx = (x, x)$. \ The adjoint operator of $W$ is  $W^* : \mathcal{H} \oplus  \mathcal{H} \to \mathcal{H}$ given by $W^*(x, y) = x+y$  for every $x, y \in \mathcal{H}$. \
The sequence ${\mathcal T} = (T_n)_{n \in \mathbb{Z}_+} \in \mathbf{B}_h(\mathcal{H}) ^{\mathbb{Z}_+}$, where $T_n = W^*B^nW$ for every $n \in\mathbb{Z}_+$, clearly satisfies the recursive relation $T_{n+2} = (a+c)T_{n+1} - (ac-b^2)T_n$. \ Moreover, $T_0=W^*W\geq 0$ and $ T_2=W^*B^2W \geq 0$. \ Also, a simple calculation gives,
$$\langle T_2x, x\rangle_\mathcal{H} \langle T_0x, x\rangle_\mathcal{H} - \langle T_1x, x\rangle_{\mathcal{H}}^{2} = (a - c)^2 \geq 0 \mbox{ for every } x \in \mathcal{S}_\mathcal{H}.$$
Therefore, from Corollary \ref{2a}, ${\mathcal T}$ is an operator moment sequence with a finitely atomic representing OVM  given by: $$E =\displaystyle \frac{1}{\lambda_1-\lambda_2}(T_1-\lambda_0T_0) \delta_{\lambda_1} +  \frac{1}{\lambda_2-\lambda_1}(T_1-\lambda_1T_0)\delta_{\lambda_2},$$  where $\lambda_1,\lambda_2$ are the two distinct roots of the quadratic polynomial $X^2-(a+c)X+ac-b^2$.
\end{ex}
\subsubsection{The case of algebraic operators} Recall that an operator $T$ is algebraic if there exists a polynomial $P$ (that we choose of minimal degree) such that $P(T)=0_\mathcal{H}.$ \ If moreover $T$ is self-adjoint, then $\mathcal{H} = \displaystyle \bigoplus_{\lambda \in  {\mathcal Z}(P)} E_{\lambda}
$, where $E_{\lambda}$ is the  eigenspace  associated with the eigenvalue $\lambda\in {\mathcal Z}(P)$. \ It is clear that ${\mathcal T}=(T^n)_{n\in\mathbb{Z}_+}$ is a recursive sequence with minimal polynomial $P$. \ We have 
\begin{prop} Let $T$ be an algebraic selfdadjoint operator and let $P$ be its minimal polynomial.  Then    ${\mathcal T}=(T^n)_{n\in\mathbb{Z}_+}$  is an operator moment sequence whose representing measure is given by $E=\displaystyle\sum_{\lambda \in  {\mathcal Z}(P)} P_\lambda.\delta_\lambda.$ \ Here $P_\lambda$ is the orthogonal projection onto   $E_\lambda$.
\end{prop}
\section{Declarations}
\subsection{Funding} \ The first-named author was partially supported by NSF grant DMS-2247167. \ The second-named author would like to thank Professor B.V. Rajamara Bhat for his clarification regarding the theory of completely positive maps. \ The last-named author was partially supported by the Arab Fund Foundation Fellowship Program. The Distinguished Scholar Award - File 1026. \ He also acknowledges the Mathematics Department at The University of Iowa for its kind hospitality during the preparation of this paper.

\subsection{Conflicts of interest/competing interests}

{\bf Non-financial interests}: \ The first-named author is on the Editorial Board of {\it Complex Analysis and Operator Theory}.

\subsection*{Data availability.}
All data generated or analyzed during this study are included in this article.


\begin{thebibliography}{99}
\bibitem{adamyan2000solution}
Adamyan, V.M., Tkachenko, I.M.: Solution of the truncated matrix Hamburger moment problem according to M.G. Krein. Operator Theory and Related Topics: Proceedings of the Mark Krein International Conference on Operator Theory and Applications, Odessa, Ukraine, August 18–22, 1997, Volume II, pp. 33–51. Springer (2000)

\bibitem{adamyan2006general}
Adamyan, V.M., Tkachenko, I.M.: General solution of the Stieltjes truncated matrix moment problem. Operator Theory and Indefinite Inner Product Spaces: Presented on the occasion of the retirement of Heinz Langer in the Colloquium on Operator Theory, Vienna, March 2004, pp. 1–22. Springer (2006)

\bibitem{akhiezer2020classical}
Akhiezer, N.I. The classical moment problem and some related questions in analysis.  Oliver \& Boyd (1965)

\bibitem{taher2001recursive}
Ben Taher, R., Rachidi, M., Zerouali, E.H.: Recursive subnormal completion and the truncated moment problem. Bull. London Math. Soc. 33(4):425–432 (2001)

\bibitem{berberian1966notes}
Berberian, S.K.: Notes on Spectral Theory. Number 5. D. Van Nostrand Co. Ltd. 118 pp. (1966)

\bibitem{berg2008matrix}
Berg, C.: The matrix moment problem. Coimbra lecture notes on orthogonal polynomials, pp. 1–57 (2008)

\bibitem{bhat2023operator}
Bhat, B., Ghatak, A., Pamula, S.K.: Operator moment dilations as block operators. arXiv:2302.13873 (2023)

\bibitem{bisgaard1994positive}
Bisgaard, T.: Positive definite operator sequences, Proc. Amer. Math. Soc. 121, 1185--1191 (1994)

\bibitem{busch2016quantum} 
Busch,P., Lahti, P., Pellonp\"a\"a, J.-P., Ylinen, K.: Quantum Measurement, volume 23, Springer (2016)

\bibitem{brandt1999positive} Brandt, H.E.: Positive operator-valued measure in quantum information processing. Amer. J. Phys., 67, 434--439 (1999)

\bibitem{cimprivc2013moment} Cimpri\v{c}, J., Zalar, A.: Moment problems for operator polynomials. J. Math. Anal. Appl. 401, 307--316
(2013)
\bibitem{curt} Curto, R.E.: Quadratically hyponormal weighted shifts. Integral Equations Operator Theory 13, 49--66 (1990)

\bibitem{cehz} Curto, R.E., Ech-charyfy, A., Idrissi, K., Zerouali, E.H.: A Recursive approach to the matrix moment problem. Preprint (2023)

\bibitem{curto1996solution}  Curto, R.E., Fialkow, L.A.: Solution of the truncated complex moment problem for flat data. Mem. Am. Math. Soc. 119(568), $x+52$ pp. (1996) 

\bibitem{curto}  Curto, R.E., Fialkow, L.A.: Recursively generated weighted shifts and the subnormal completion problem, II. Integral Equations and Operator Theory 18, 369--426 (1994)

\bibitem{dubin2014operator} Dubin, D., Kiukas, J., Pellonpää, J.-P., Ylinen, K.: Operator integrals and sesquilinear forms, J. Math. Anal. Appl. 413, 250--268 (2014)

\bibitem{guardeno2001matrix}
Duran, A.J., Rodr\'iguez, P.L.: The matrix moment problem. In Margarita Mathematica: en memoria de José Javier (Chicho) Guadalupe Hernández, pp. 333–348. Universidad de La Rioja (2001)

\bibitem{dyukarev2009distinguished}
Dyukarev, Y.M., Fritzsche, B., Kirstein, B., Mädler, C., Thiele, H.C.: On distinguished solutions of truncated matricial Hamburger moment problems. Complex Analysis and Operator Theory 3, 759–-834 (2009)

\bibitem{ghatage1976subnormal}  Ghatage, P.: Subnormal shifts with operator-valued weights. Proc. Amer. Math. Soc. 57, 107--108 (1976)

\bibitem{ivan} Ivanovski, N.: Subnormality of operator-valued weighted shifts,  Ph.D. dissertation, Indiana University (1973)

\bibitem{kims} Kimsey, D.P.: An operator-valued generalization of Tchakaloff's theorem. J. Funct. Anal. 266:1170--1184 (2014)

\bibitem{KimWoe} Kimsey, D., Woerdeman, H: The truncated matrix-valued $K$--moment problem on $\mathbf{R}^d$, $\mathbf{C}^d$, and $\mathbf{T}^d$. Trans. Amer. Math. Soc. 365, 5393--5430 (2013)

\bibitem{kimsey2022solution}
Kimsey, D.P., Trachana, M.: On a solution of the multidimensional truncated matrix-valued moment problem. Milan J. Math. 90(1):17–101 (2022)

\bibitem{krein1949fundamental}
Krein, M.G.: Fundamental aspects of the representation theory of hermitian operators with deficiency index $(m, m)$. Amer. Math. Soc. Translations, series 2(97):75–143 (1949)

\bibitem{krein1947fundamental}
Krein, M.G., Krasnosel'skii, M.A.: Fundamental theorems on the extension of hermitian operators and certain of their applications to the theory of orthogonal polynomials and the problem of moments. Uspekhi Matematicheskikh Nauk 2(3):60–106 (1947)

\bibitem{lambert1976subnormality} Lambert, A.: Subnormality and weighted shifts. J. London Math. Soc. 2(3):476--480 (1976)

\bibitem{lemnete2014hamburger}
Lemnete-Ninulescu, L.: Hamburger and Stieltjes moment problems for operators. International Scholarly Research Notices (2014)

\bibitem{madler2023truncated} Mädler, C., Schmüdgen, K.: On the truncated matricial moment problem. I. Journal of Mathematical Analysis and Applications, 128569 (2024).

\bibitem{mlak1978dilations}
Mlak, M.:Dilations of Hilbert space operators (general theory).  Instytut Matematyczny Polskiej Akademi Nauk(Warszawa) (1978)

\bibitem{moretti2017spectral}
Moretti, V.: Spectral Theory and Quantum Mechanics, 2nd revised and enlarged edition. Springer International Publishing (2017)

\bibitem{Neumark}
Neumark, M.: Spectral functions of a symmetric operator. Izv. Akad. Nauk SSSR Ser. Mat. 121(4:3):27--318 (1940)

\bibitem{paulsen2002completely}
Paulsen, V.I.: Completely bounded maps and operator algebras. Number 78. Cambridge University Press (2002)

\bibitem{pietrzyckistochelJFA} Pietrzycki, P., Stochel, J.: Subnormal nth roots of quasinormal operators are quasinormal. J. Funct. Anal. 280, 109001 (2021)

\bibitem{pietrzycki2022two}
Pietrzycki, P., Stochel, J.: Two-moment characterization of spectral measures on the real line. Canadian J. Math. pp. 1--24 (2022)

\bibitem{schmudgen2012unbounded}
Schm\"udgen. K.: Unbounded Self-adjoint Operators on Hilbert Space. Volume 265. Springer Science and Business Media (2012)

\bibitem{schmudgen2017moment}
Schm\"udgen, K.: The Moment Problem. Volume 277. Springer (2017)
\bibitem{stampfli1966weighted}
Stampfli, J.: Which weighted shifts are subnormal? Pacific J. Math. 17, 367--379 (1966)
  
\bibitem{Sto} Stochel, J.: Solving the truncated moment problem solves the moment problem. Glasgow Math. J.  43, 335--341 (2001)

\bibitem{Szafran1} Szafraniec, F.H.: Moments from their very truncations. Contemporary Mathematics 435, 363--370 (2007)

\bibitem{Szafran} Szafraniec, F.H.: Sesquilinear selection of elementary spectral measures and subnormality. In
Elementary Operators and Applications, Proceedings, Blaubeuren bei Ulm (Deutschland), June 9–12, 1991.  World Scientific, Singapore, 243--248 (1992)

\bibitem{Nagy1952}
Sz-Nagy, B.: A moment problem for self-adjoint operators.  Acta Mathematica Academiae Scientiarum Hungarica 3, 285-–293 (1952)

\bibitem{tcha} Tchakaloff, V.: Formules de cubature m\'ecanique \`a coefficients non n\'egatifs. Bull. Sci. Math. 81, 123--134 (1957)

\bibitem{vasilescu} Vasilescu, F.-H.: Moment problems for multi-sequences of operators. J. Math. Anal. Appl. 219, 246--259 (1998) 
\end{thebibliography}
\end{document}